\documentclass[11pt]{amsart}
\usepackage{amssymb,amsmath,amsthm,mathrsfs,verbatim}
\usepackage[all, knot]{xy}
\usepackage{rotating}
\usepackage{cancel}
\newtheorem{theorem}{Theorem}
\newtheorem{lemma}[theorem]{Lemma}
\newtheorem{corollary}[theorem]{Corollary}

\newtheorem{proposition}[theorem]{Proposition}
\newtheorem{definition}[theorem]{Definition}

\theoremstyle{remark}

\newtheorem*{remark}{Remark}
\newtheorem*{remarks}{Remarks}

\numberwithin{theorem}{section} \numberwithin{equation}{section}
\numberwithin{figure}{section}

\newcommand{\cM}{\mathcal{M}}

\newcommand{\Fe}{\mathcal{F}}
\newcommand{\GG}{\mathcal{G}}
\newcommand{\CC}{\mathcal{C}}
\newcommand{\wt}{\kappa}
\newcommand{\FF}{\mathcal{F}}
\newcommand{\HH}{\mathbb{H}}
\newcommand{\HHcusp}{\HH_{\frac{3}{2}-k,d}^{\text{cusp}}}
\newcommand{\HHcuspwt}{\HH_{\wt}^{\text{cusp}}}
\newcommand{\QQ}{\mathcal{Q}}

\newcommand{\R}{\mathbb{R}}
\newcommand{\C}{\mathbb{C}}
\newcommand{\Tr}{{\text {\rm Tr}}}

\newcommand{\Z}{\mathbb{Z}}
\newcommand{\N}{\mathbb{N}}

\newcommand{\SL}{{\text {\rm SL}}}

\newcommand{\sgn}{\operatorname{sgn}}
\newcommand{\pr}{\operatorname{pr}}
\newcommand{\shin}{\mathscr{S}^*}
\newcommand{\shim}{\mathscr{S}}
\newcommand{\reg}{\operatorname{reg}}

\newcommand{\W}{\mathcal{W}}

\newcommand{\iso}{\cong}

\newcommand{\re}{\textnormal{Re}}

\def\H{\mathbb{H}}
\newcommand{\Hcusp}{H_{2-2k}^{\text{cusp}}}
\newcommand{\Hcuspwt}{H_{\wt}^{\text{cusp}}}
\newcommand{\Za}{\mathfrak{Z}}

\newcommand{\PP}{\mathcal{P}}

\newcommand{\Eis}{h_0}
\renewcommand{\SS}{\mathbb{S}}

\newcommand{\MM}{\mathbb{M}}
\newcommand{\DD}{\mathcal{D}}
\newcommand{\Der}{\DD_{d,D}^{k-\frac{1}{2}}}

\newcommand{\JD}{\mathbb{J}_{k+\frac{1}{2}}^{D}}
\newcommand{\Jd}{\mathbb{J}_{\frac{3}{2}-k}^{d}}

\newcommand{\Sperp}{S_{2k}^{\perp,0}}
\newcommand{\Swtperp}{S_{\wt}^{\perp,0}}
\newcommand{\Swtperpdual}{S_{2-\wt}^{\perp,0}}
\newcommand{\SSwtperp}{\SS_{\wt}^{\perp,0}}
\newcommand{\SSperp}{\SS_{k+\frac{1}{2}}^{\perp,0}}
\newcommand{\SSperpD}[1]{\SS_{k+\frac{1}{2},#1}^{\perp,0}}
\begin{document}

\title[Shintani lifts]{Shintani lifts and fractional derivatives for harmonic weak Maass forms}

\author{Kathrin Bringmann} 
\address{Mathematical Institute\\University of
Cologne\\ Weyertal 86-90 \\ 50931 Cologne \\Germany}
\email{kbringma@math.uni-koeln.de}
\author{Pavel Guerzhoy}
\address{Department of Mathematics\\ University of Hawaii\\ Honolulu, HI 96822-2273}
\email{pavel@math.hawaii.edu}
\author{Ben Kane}
\address{Mathematical Institute\\University of
Cologne\\ Weyertal 86-90 \\ 50931 Cologne \\Germany}
\email{bkane@math.uni-koeln.de}
\date{\today}
\thanks{The research of the first author was supported by the Alfried Krupp Prize for Young University Teachers of the Krupp Foundation and by the Deutsche Forschungsgemeinschaft (DFG) Grant No. BR 4082/3-1.  The research of the second author is supported by  Simons Foundation Collaboration Grant.}

\begin{abstract}
In this paper, we construct Shintani lifts from integral weight weakly holomorphic modular forms to half-integral weight weakly holomorphic modular forms.  Although defined by different methods, these coincide with the classical Shintani lifts when restricted to the space of cusp forms.  As a side effect, this gives the coefficients of the classical Shintani lifts as new cycle integrals.  This yields new formulas for the $L$-values of Hecke eigenforms.  When restricted to the space of weakly holomorphic modular forms orthogonal to cusp forms, the Shintani lifts introduce a definition of weakly holomorphic Hecke eigenforms.  Along the way, auxiliary lifts are constructed from the space of harmonic weak Maass forms which yield a ``fractional derivative'' from the space of half-integral weight harmonic weak Maass forms to half-integral weight weakly holomorphic modular forms.  This fractional derivative complements the usual $\xi$-operator introduced by Bruinier and Funke.
\end{abstract}

\maketitle

\section{Introduction and statement of results}

In this paper, we introduce an extension of Shintani's lifts \cite{Shintani} from integral weight weakly holomorphic modular forms to half-integral weight weakly holomorphic modular forms.  The resulting lifts are essentially equal to the classical lifts when one restricts to cusp forms, but our construction differs from that of Shintani, yielding a new interpretation of his lifts.  

A number of remarkable theorems have been proven using the relationship between integral and half-integral weight modular forms.  In celebrated work of Waldspurger \cite{Waldspurger}, the interplay between these two spaces was used to solve open problems in both spaces.  On the integral weight side, it had been (empirically) observed that central values of $L$-functions of integral weight Hecke eigenforms were basically squares.  On the other hand, the coefficients of half-integral weight eigenforms were shown by Shintani \cite{Shintani} to be certain cycle integrals of integral weight forms.  Waldspurger fused these two statements into one coherent theory by proving that the central value of the $L$-function for integral weight Hecke eigenforms is proportional to the square of a coefficient of a certain half-integral weight form.  The explicit constant of proportionality is the main theme of the famous Kohnen-Zagier formula \cite{KohnenZagier}.  By proving that the constant of proportionality is positive, Kohnen and Zagier showed the non-negativity of central values of $L$-series.  To give another example, Iwaniec \cite{Iwaniec} proved bounds for the size of the Fourier coefficients of half-integral weight modular forms.  Combining these bounds with the results of Waldspurger, Iwaniec obtained sub-convexity bounds for central values.  Conversely, recent developments in analytic number theory have led to improvements in the sub-convexity bounds of central $L$-values and accordingly sharpened the bounds of Iwaniec by using this connection in the reverse direction (for example, see \cite{BHM}).

Having demonstrated the utility of the connection between integral and half-integral weight modular forms, we now briefly turn to the history of the construction of lifts between them.  The interrelation between integral and half-integral weight modular forms was initiated by Shimura \cite{Shimura}, who constructed  a lift $\shim_{n}$ (for every $n\in \N$ squarefree)
 from half-integral weight cusp forms to integral weight holomorphic forms.   The surprising connection between these two spaces led to a flurry of activity.  Among the exciting work that followed, Shintani \cite{Shintani} defined a lift $\shin_d$ (for every fundamental discriminant $d$ and integer $k>1$ satisfying $(-1)^kd>0$) from integral to half-integral weight cusp forms whose coefficients are related to cycle integrals of the integral weight cusp forms which we now briefly describe.  There is a one-to-one correspondence between $Q$ in the set $\QQ_{\Delta}$ of binary quadratic forms of non-square discriminant $\Delta>0$ and semicircles $S_Q$ whose real endpoints are the roots of $Q$ (see \eqref{eqn:SQdef}).  Moreover, the automorphs of $Q$ in $\SL_2(\Z)$ are an infinite cyclic group $\Gamma_Q$.  Defining $C_Q:=S_Q/\Gamma_Q$, for every function $F$ satisfying weight $2\wt\in 2\Z$ modularity we define the \begin{it}cycle integral\end{it}
$$
\CC\left(F;Q\right):=\Delta^{\frac{1-
\wt}{2}}\int_{C_Q} F(\tau) Q(\tau,1)^{\wt-1}d\tau.
$$
Taking traces over all $\SL_2(\Z)$ equivalence classes twisted by the genus characters $\chi$ defined in \eqref{eqn:genusdef}, Shintani's lifts are then defined for $f\in S_{2k}$ (cf. (8) of \cite{KohnenFourier}) by 
\begin{equation}\label{eqn:shindef}
\shin_d\left(f\right)(\tau):=\sum_{\delta:\delta d>0}\left(\left(\delta d\right)^{\frac{k-1}{2}}\sum_{Q\in \SL_2(\Z)\backslash\QQ_{\delta d}} \chi(Q)\CC\left(f;Q\right)\right) q^{|\delta|},
\end{equation}
where $q:=e^{2\pi i \tau}$.
\rm
It was later shown by Kohnen and Zagier \cite{KohnenZagier} that Shintani's lift is adjoint to Shimura's lift.

In order to provide an isomorphism between these two spaces, Kohnen \cite{Kohnen} defined for $k\geq 0$ a distinguished subspace $\SS_{k+\frac{1}{2}}$ of weight $k+\frac{1}{2}$ cusp forms (now known as the \begin{it}Kohnen plus space\end{it}).   Since the Shimura and Shintani maps both commute with the Hecke operators, there exists for every fundamental discriminant $d$ with $(-1)^kd>0$ and weight $2k$ Hecke eigenform $f$ a constant $c_{f,d}\in \C$ (possibly zero) for which
$$
\shim_{d}\circ\shin_d\left(f\right) = c_{f,d} f.
$$
Kohnen \cite{Kohnen} further proved the existence of a linear combination of finitely many Shimura lifts which yields an isomorphism between $S_{2k}$ (classical weight $2k$ cusp forms) and $\SS_{k+\frac{1}{2}}$.  

In order to extend Shintani's lifts to weakly holomorphic modular forms, we require pairs of fundamental discriminants $d$ and $D$ which, as we assume throughout, satisfy $(-1)^kd>0$ and $(-1)^kD<0$.  Roughly speaking, the choice of $d$ determines the cuspidal projection of our Shintani lift $\shin_{d,D}$ (explicitly defined in \eqref{eqn:shindH}), while $D$ contributes the principal part of the lift, which governs the growth of the function towards the cusp $i\infty$.  More precisely, using an extension of the Petersson inner product to weakly holomorphic modular forms, $D$ provides the projection to the space orthogonal to cusp forms.  

The lifts $\shin_{d,D}$ are constructed via intermediary lifts on non-holomorphic modular forms known as harmonic weak Maass forms (see Section \ref{sec:Maass} for the definition).  When restricted to the subspace of cusp forms, they are furthermore related to the classical Shintani lifts via 
\begin{equation}\label{eqn:shindD}
\shin_{d,D}= \frac{1}{3}(-1)^{\left\lfloor\frac{k}{2}\right\rfloor} 2^{k-1} \shin_d.
\end{equation}
By computing the Fourier coefficients of the harmonic weak Maass forms in two different ways, one obtains a striking identity between Shintani's classical cycle integrals and cycle integrals of weight $0$ non-holomorphic modular forms.  To state the result, we need the antiholomorphic differential operator $\xi_{\wt}:=2iy^{\wt}\overline{\frac{\partial}{\partial\overline{\tau}}}$ and the Maass raising operator $R_{\wt}$ defined in \eqref{eqn:raisingdef}.  While $\xi_{\wt}$ maps harmonic weak Maass forms of weight $\wt$ to weakly holomorphic modular forms of weight $2-\wt$, $R_{\wt}$ maps non-holomorphic modular forms of weight $\wt$ to non-holomorphic modular forms of weight $\wt+2$.  We denote the $n$-th repeated iteration by $R_{\wt}^n$.  The following theorem yields a new construction of Shintani's lifts.
\begin{theorem}\label{thm:Zaddef2}
If $\cM$ is a weight $2-2k$ harmonic weak Maass form for which $\xi_{2-2k}(\cM)$ is a cusp form, then for every pair of positive discriminants $d,\delta$ for which $d\delta$ is not a square we have
\begin{equation}\label{eqn:Zaddef}
\sum_{Q\in \SL_2(\Z)\backslash\QQ_{\delta d}} \chi(Q)\CC\left(\xi_{2-2k}(\cM);Q\right)=C_{k}\sum_{Q\in \SL_2(\Z)\backslash\QQ_{\delta d}} \chi(Q)\CC\left(R_{2-2k}^{k-1}\left(\cM\right);Q\right),
\end{equation}
where $C_{k}:=-\frac{3\Gamma\left(\frac{k+1}{2}\right)}{2^{k-1}\Gamma\left(k-\frac{1}{2}\right)\Gamma\left(\frac{k}{2}\right)}$.
\end{theorem}
\begin{remark}
The twisted traces on the right hand side of \eqref{eqn:Zaddef} were studied for the excluded case $k=1$ in \cite{DIT},  where cycle integrals for meromorphic and non-holomorphic forms were first considered.  Cycle integrals of this type also occur in \cite{DITAbh}.  Methods for the case $k=1$ were developed in \cite{BFI} which we expect could be extended to yield the coefficients in Theorem \ref{thm:Zaddef2} when $\delta d$ is a square.
\end{remark}
The lifts $\shin_{d,D}$ also satisfy a number of remarkable properties on the entire space of weakly holomorphic modular forms.
\begin{theorem}\label{thm:proportion}
\noindent

\noindent
\begin{enumerate}
\item
The lifts $\shin_{d,D}$ all commute with the Hecke operators.
\item
The lifts $\shin_{d,D}$ all map cusp forms to cusp forms and preserve orthogonality to cusp forms.
\end{enumerate}
\end{theorem}

In addition to yielding the proportionality in Theorem \ref{thm:Zaddef2}, the construction of pairs of lifts from integral to half-integral weight harmonic weak Maass forms plays a key role in extending the Shintani lifts to include weakly holomorphic modular forms.  The resulting pairs of Zagier lifts, $\Za_d$ and $\Za_D$, build upon a construction of lifts \cite{DukeJenkins} between spaces of weakly holomorphic modular forms.  The lifts $\Za_d$ \eqref{eqn:Zadpp2} map harmonic weak Maass forms of weight $2-2k$ to harmonic weak Maass forms of weight $\frac{3}{2}-k$, while the lifts $\Za_D$ \eqref{eqn:ZaDdef} map weight $2-2k$ harmonic weak Maass forms directly to weight $k+\frac{1}{2}$ weakly holomorphic modular forms.  Using the connection of Waldspurger \cite{Waldspurger} between coefficients of half-integral weight Hecke eigenforms and central values of integral weight Hecke eigenforms, Theorem \ref{thm:proportion} (2) leads to the following classification of the vanishing of the central $L$-value $L\left(f,\chi_d,k\right)$ of $f\in S_{2k}$ twisted by the quadratic character $\chi_d$.  
\begin{corollary}\label{cor:weakly}
Suppose that a cusp form $f$ is a weight $2k$ Hecke eigenform and $\mathcal{M}$ is a weight $2-2k$ harmonic weak Maass form satisfying $\xi_{2-2k}\left(\mathcal{M}\right)=f$.  Then $\Za_d\left(\mathcal{M}\right)$ is weakly holomorphic if and only if $L\left(f,\chi_d,k\right)=0$.
\end{corollary}
After completing this paper, Bruinier informed the authors that  Zagier lifts also appear as theta lifts.  In particular, Alfes \cite{Alfes}, motivated by algebraicity questions, has obtained overlapping partial results extending the Zagier lifts in certain cases.  The cases of Corollary \ref{cor:weakly} with $k$ even are contained in \cite{Alfes}.  

Since the obstruction to surjectivity of the classical Shintani maps yield interesting arithmetic information about vanishing of $L$-values, it is worthwhile to determine the image of the Shintani lifts when restricting to the subspace $\Sperp$ consisting of those forms orthogonal to cusp forms with vanishing constant term.  For this, we decompose the orthogonal complement of cusp forms $\SSperp$ in weight $k+\frac{1}{2}$ with vanishing constant terms into subspaces $\SSperpD{D}$ consisting of those forms whose principal parts are supported in the square class $-|D|m^2$ ($m\in \N_0$).
\begin{theorem}\label{thm:shinimage}
For each pair $d,D$, the restriction of $\shin_{d,D}$ to $\Sperp$ yields a bijection from $\Sperp$ to $\SSperpD{D}$.
\end{theorem}
We have seen that $\Za_d$ and $\Za_D$ play an important role in our new construction of the Shintani lift.  To better understand these lifts, we next determine the image of the space $\Hcusp$ of harmonic weak Maass forms for which $\xi_{2-2k}$ maps to cusp forms.  As in Theorem \ref{thm:shinimage}, it is natural to decompose the space of weight $\frac{3}{2}-k$ harmonic weak Maass forms into subspaces $\HHcusp$ consisting of those forms whose principal parts are supported in the $-|d|m^2$ square class ($m\in \N$).  In particular, elements of $\HHcusp$ map to cusp forms under the $\xi$-operator.

\begin{theorem}\label{thm:Zaharmonic}
\noindent

\noindent
\begin{enumerate}
\item
Each lift $\Za_D$ is an isomorphism from $\Hcusp$ to $\SSperpD{D}$.
\item
Every  map $\Za_d$ is an isomorphism from $\Hcusp$ to $\HHcusp$.
\end{enumerate}
\end{theorem}
\begin{remark}
Theorem \ref{thm:proportion} further allows us to use the maps $\Za_d$ to give an alternate proof of the famous classical theorem of Kohnen \cite{Kohnen}, proving the existence of the Shimura isomorphism mentioned above.
\end{remark}

While the Zagier lifts $\Za_d$ give a new criterion to determine vanishing of central $L$-values, the lifts $\Za_D$ may be used to build a theory of half-integral weight weakly holomorphic Hecke eigenforms.  This parallels a construction of integral weight weakly holomorphic Hecke eigenforms by the second author \cite{GuerzhoyHecke}.  He builds such forms by using the differential operator $\DD^{2k-1}:=\left(\frac{1}{2\pi i}\frac{\partial}{\partial \tau}\right)^{2k-1}$ to define a distinguished finite dimensional quotient space.  A key observation is the fact that $\Sperp$, which is the image under $\DD^{2k-1}$ of the subspace $S_{2-2k}^!$ of weakly holomorphic modular forms with vanishing constant term, is preserved by the Hecke operators. 

Trying to follow this construction directly in the half-integral weight case immediately runs into a number of roadblocks.  Although we have the operator $\xi_{\frac{3}{2}-k}$, we are missing a complementary operator which plays the role of $\DD^{2k-1}$.  Without the ability to act by a ``fractional derivative,'' $\DD$ does not aid in the construction of such an operator.  Moreover, even once an operator from the space of weight $\frac{3}{2}-k$ harmonic weak Maass forms in Kohnen's plus space to $\SSperp$ is found, one needs to determine a distinguished subspace which is preserved under the Hecke operators.  To overcome these problems, we define the subspace
\begin{equation}\label{eqn:JDdef}
\JD:=\Za_D\left(S_{2-2k}^!\right),
\end{equation}
which is Hecke stable and serves the role of $\DD^{2k-1}\left(S_{2-2k}^!\right)$  in our construction of half-integral weight weakly holomorphic Hecke eigenforms.  We call a weight $k+\frac{1}{2}$ weakly holomorphic modular form $F$ a \begin{it}Hecke eigenform\end{it} if for every $n\in \N$ there exists $\lambda_n\in \C$ and $F_n\in\bigoplus_{D}\JD$ such that 
\begin{equation}\label{eqn:HE}
F\big| T_n = \lambda_n F + F_n,
\end{equation}
where $T_n$ is the $n$-th Hecke operator.  Since there are already Hecke eigenforms for $\MM_{k+\frac{1}{2}}$ and $\JD\subseteq \SSperp$ by Theorem \ref{thm:Zaharmonic}, it suffices to study Hecke eigenforms on the subspace $\SSperp$.

The next theorem shows that
\begin{equation}\label{eqn:Ddecomp}
\SSperp\Big/\bigoplus_{D}\JD = \bigoplus_D \left(\SSperpD{D}\Big/\JD\right)
\end{equation}
is spanned by Hecke eigenforms.  However, contrasting the integral weight case, the eigenspaces are infinite dimensional.  We overcome this by looking separately at each component of the decomposition on the right hand side of \eqref{eqn:Ddecomp}, which turn out to be finite dimensional. 
\begin{theorem}\label{thm:HeckeD}
The spaces $\SSperpD{D}\Big/\JD$ and $\MM_{k+\frac{1}{2}}$ are isomorphic as Hecke modules.
\end{theorem}
As mentioned in the discussion preceding Theorem \ref{thm:HeckeD}, there is interest in determining a natural operator on weight $\frac{3}{2}-k$ harmonic weak Maass forms which complements $\xi_{\frac{3}{2}-k}$.  Paralleling the integral weight operator $\DD^{2k-1}$, this new operator should map $\HHcusp$ to $\SSperp$.  Theorem \ref{thm:Zaharmonic} leads to such a natural Hecke equivariant isomorphism defined by
\begin{equation}\label{eqn:Shimura}
\Der:=\Za_D\circ \Za_d^{-1}:\HHcusp \overset{\sim}{\longrightarrow} \SSperpD{D}.
\end{equation}
We shall see that $\Der$ parallels $\DD^{2k-1}$ in a number of ways, leading us to call $\Der$ the \begin{it}$(d,D)$-th fractional derivative (of weight $k-\frac{1}{2}$).\end{it}  

To describe another important property of $\DD^{2k-1}$, we note that on Fourier expansions 
$$
\DD^{2k-1}\left(\sum_{n\geq 1} a(n)q^n\right) = \sum_{n\geq 1} n^{2k-1}a(n) q^n.
$$
This link is important for arithmetic applications.  Moreover, uniquely define $\Fe_m\in \Hcusp$ ($m\in \N$) by $\Fe_m(\tau)=q^{-m}+O(1)$ and $g_m\in \Sperp$ by $g_m(\tau)=q^{-m}+O(q)$.  The coefficients of these functions satisfy a so-called Zagier duality \cite{ZagierTraces}.  This duality may be proven using the $\DD^{2k-1}$-operator via the identity \cite{GuerzhoyGrids}
\begin{equation}\label{eqn:DDduality}
\DD^{2k-1}\left(\Fe_m\right) = -m^{2k-1} g_m.
\end{equation}

To explain the statement of duality, uniquely define $\GG_d\in \HHcusp$ by $\GG_d(\tau)=q^{-|d|}+O(1)$ and $\widetilde{h}_D\in \SSperp$ by $\widetilde{h}_D(\tau)=q^{-|D|}+O(q)$.  Denote the $D$-th (resp. $d$-th) coefficient of $\GG_d$ (resp. $\widetilde{h}_D$) by $a(d,D)$ (resp. $c(D,d)$).  Generalizing Zagier's original duality statement \cite{ZagierTraces}, the first author and Ono \cite{BringmannOnoMathAnn} proved that
\begin{equation}\label{eqn:Zagduality}
a(d,D)=-c(D,d).
\end{equation}
Denote the $d$-th coefficient of $h_D:=\Der\left(\GG_d\right)$ by $b(D,d)$.  Paralleling \eqref{eqn:DDduality}, the coefficients of $\GG_d$ and $h_D$ satisfy duality.
\begin{theorem}\label{thm:Derduality}
The fractional derivatives $\Der$ are Hecke equivariant isomorphisms from $\HHcusp$ to $\SSperpD{D}$.  Moreover, one has 
\begin{equation}\label{eqn:Derduality}
a(d,D)=-b(D,d).
\end{equation}
\end{theorem}
\begin{remarks}
\noindent

\noindent
\begin{enumerate}
\item
By showing that $h_D=\widetilde{h}_D$, one could use the maps $\Der$ to give an alternate proof of the duality \eqref{eqn:Derduality}.
\item
Although one cannot get all of the coefficients of $h_D$ from a single fractional derivative, they can be obtained by taking $\Der$ for all $d$.  In order to do so, one must further extend the definition of the Zagier lifts to non-fundamental discriminants, which can be achieved by using the Hecke operators.  See \eqref{eqn:shingen} for more details.
\end{enumerate}
\end{remarks}

The paper is organized as follows.  The basic properties of harmonic weak Maass forms are recalled in Section \ref{sec:Maass}.  The definitions of the Zagier lifts are the main goal of Section \ref{sec:Zagier}.  In Section \ref{sec:Poincare}, a well-known set of Poincar\'e series, which span the space of harmonic weak Maass forms, is constructed.  The relationship between the Zagier and Shintani lifts is discussed in Section \ref{sec:ZagShin}.   The modularity properties of the image of the Zagier lifts are determined in Section \ref{sec:ortho}.  We define the generalized Shintani lift for weakly holomorphic modular forms in Section \ref{sec:weakShin} and investigate Hecke eigenforms for half-integral weight weakly holomorphic modular forms in Section \ref{sec:Hecke}.  Finally, we investigate the fractional derivative $\Der$ in Section \ref{sec:halfderiv}.

\section*{Acknowledgements}
The authors thank Zagier for assistance with computer calculations related to Theorem \ref{thm:Zaddef2}.  We also thank the referees for carefully reading the paper and for their comments.  Finally, the authors thank SoYoung Choi for pointing out a correction to the formula \eqref{eqn:shindef} after the article's publication.

\section{Harmonic weak Maass forms}\label{sec:Maass}
Here we recall the basic facts about the theory of harmonic weak Maass forms.  For a good reference on this theory, see \cite{BruinierFunke}.  Throughout, we write $\tau \in \H$ as  $\tau=x+iy$ with $x, y\in \R$, $y>0$ and
let $\wt\in \frac{1}{2}\Z$.  We define the weight $\wt$ \textit{hyperbolic Laplacian} by
\begin{equation*}
\Delta_{\wt} := -y^2\left( \frac{\partial^2}{\partial x^2}+
\frac{\partial^2}{\partial y^2}\right) + i\wt y\left(
\frac{\partial}{\partial x}+i \frac{\partial}{\partial y}\right).
\end{equation*}

For notational ease, we set
\begin{equation}\label{eqn:Gammadef}
\Gamma:=
\begin{cases} 
\SL_2(\Z)&\text{if }\wt\in \Z,\\
\Gamma_0(4)&\text{if }\wt\in \frac{1}{2}\Z\setminus \Z.\end{cases}
\end{equation}
 For $\gamma=\left(\begin{smallmatrix}a&b\\c&d\end{smallmatrix} \right)\in \Gamma$ and any function $g:\H\to \C$, we define
\begin{equation*}\label{slash}
g\big|_{\wt}\gamma (\tau):=  j(\gamma,\tau)^{-2\wt} g \left( \frac{a\tau +b}{c\tau+d} \right),
\end{equation*}
where
\begin{equation*}\label{eqn:j}
j(\gamma, \tau):=
\begin{cases}\sqrt{c\tau+d} &\text{if } \wt \in \Z,\\ 
\left(\frac{c}{d}\right) \varepsilon_d^{-1} \sqrt{c\tau+d} & \text{if } \wt \in \frac{1}{2}\Z\setminus \Z.
\end{cases}
\end{equation*}
Here $\left(\frac{c}{d}\right)$ is the usual Kronecker-Jacobi symbol and for $d\equiv 1\pmod{4}$ we denote $\varepsilon_d=1$, while for $d\equiv 3\pmod{4}$ we have $\varepsilon_d=i$.
\begin{definition}
A {\it{weak Maass form}} of weight $\wt$ and Laplace eigenvalue $\lambda$ is a smooth function $\cM:\H\to\C$ satisfying:
\begin{enumerate}
\item[(i)] $\cM\big|_{\wt} \gamma = \cM$ for all $\gamma \in \Gamma$,
\item[(ii)] $\Delta_{\wt}\left(\cM\right)=\lambda \cM $,
\item[(iii)] $\cM$ has at most linear exponential growth at each cusp of $\Gamma$.
\end{enumerate}
If $\lambda=0$, then we call $\cM$ \begin{it}harmonic.\end{it}
\end{definition}

The operator $\xi_{\wt}=2iy^{\wt} \overline{\frac{\partial}{\partial \overline{\tau}}}$ maps harmonic weak Maass forms of weight $\wt$ to weakly holomorphic modular forms of weight $2-\wt$ and plays a fundamental role in the theory of harmonic weak Maass forms.  Its kernel is precisely the subspace $M_{\wt}^!$ consisting of weight $\wt$ weakly holomorphic modular forms.  We denote the space of weight $\wt$ harmonic weak Maass forms which map to cusp forms under $\xi_{\wt}$ by $\Hcuspwt$.  Bruinier and Funke \cite{BruinierFunke} have shown that the map $\xi_{\wt}$ is surjective onto $S_{2-\wt}$.  For $\wt \le \frac{1}{2}$, $\cM\in \Hcuspwt$ has a Fourier expansion of the shape
\begin{equation}\label{eqn:fourier}
\cM(\tau) =\frac{1}{\Gamma(1-\wt)}\sum_{\begin{subarray}{c}n<0 \end{subarray}} c_{\cM}^-(n) \Gamma\left(1-\wt; 4\pi |n| y \right)q^n + \sum_{n \gg -\infty} c_{\cM}^+(n) q^n,
\end{equation}
where, for $y>0$, $\Gamma\left(s;y\right):=\int_{y}^{\infty} t^{s-1} e^{-t}dt$ is the \begin{it}incomplete $\Gamma$-function\end{it}.  We call $\cM^+(\tau) := \sum_{n\gg -\infty} c_{\cM}^+(n)q^n$ the {\it{holomorphic part}} of $\cM$ and $\cM^-:=\cM- \cM^+$ the {\it non-holomorphic part} of $\cM$.  The sum $\sum_{n< 0} c_{\cM}^+(n)q^n$  is called the \begin{it}principal part\end{it} of $\cM$.  The action of the $\xi$-operator on the Fourier expansion \eqref{eqn:fourier} is given by
\begin{equation}\label{eqn:xifourier}
\xi_{\wt}\left(\cM(\tau)\right)=\frac{1}{\Gamma(1-\wt)}\sum_{n>0} \left(4\pi n\right)^{1-\wt} \overline{c_{\cM}^{-}(-n)}q^{n}.
\end{equation}

Whenever $\wt$ is a negative integer, the operator $\DD^{1-\wt}=\left(\frac{1}{2\pi i} \frac{\partial}{\partial \tau}\right)^{1-\wt}$ injects $\Hcuspwt$ into $S_{2-\wt}^!,$ the subspace of $M_{2-\wt}^!$ consisting of those forms with vanishing constant terms.  As in the introduction, we let $\HHcuspwt\subset \Hcuspwt$ denote the subspace of those half-integral weight harmonic weak Maass forms satisfying the plus space conditions of Kohnen, i.e., those forms whose Fourier expansions are supported only on those $n\in \Z$ for which $(-1)^{\wt-\frac{1}{2}} n\equiv 0,1\pmod{4}$.  Furthermore, we let $\MM_{\wt}^!$ denote the subspace of $\HHcuspwt$ consisting of weakly holomorphic modular forms and $\SS_{\wt}^!\subseteq \MM_{\wt}^!$ denote the subspace precisely containing those weakly holomorphic modular forms for which the constant term of the Fourier expansion vanishes at every cusp.  We decompose $\HHcuspwt$ (resp. $\MM_{\wt}^!$) into the subspaces $\HH_{\wt,\delta}^{\text{cusp}}$ (resp. $\MM_{\wt,\delta}^!$) consisting of those forms whose principal parts are supported in the $-|\delta|m^2$ square class ($m\in \N$).  We use $\pr$ to denote the usual extension of Kohnen's orthogonal projection operator \cite{KohnenFourier} to $\Hcuspwt$.

Since the Hecke operators commute, we only give the $p$-th (resp. $p^2$-th Hecke operator) if $\wt\in \Z$ (resp. $\wt\in \frac{1}{2}\Z$).  For a translation invariant function $f(\tau) = \sum_{n\in \Z} c(y;n) e^{2\pi i nx}$, $\wt\in \Z$, and a prime $p$, the $p$-th \begin{it}Hecke operator\end{it} is defined by
\begin{equation}\label{eqn:Heckeint}
f\big|_{\wt}T(p)(\tau):= p^{\alpha_{\wt}}\sum_{n\in \Z} \left(c\left(\frac{y}{p};np\right)+p^{\wt-1} c\left(py;\frac{n}{p}\right)\right)e^{2\pi i n x}.
\end{equation}
Here we have renormalized the classical Hecke operators with the $p$-power
$$
\alpha_{\wt}:=
\begin{cases}
1-\wt & \text{if }\wt<0,\\
0 &\text{otherwise},
\end{cases}
$$
so that the operators in this paper commute with the Hecke operators.  We omit $\wt$ when it is clear from the context.  Similarly, for $\wt=\lambda+\frac{1}{2}\in \frac{1}{2}\Z\setminus\Z$, the \begin{it}$p^2$-th Hecke operator\end{it} is defined by
\begin{multline}\label{eqn:Heckedef}
f\big|_{\wt} T\left(p^2\right)(\tau):=p^{2\alpha_{\wt}}\sum_{n\in \Z}\left(c\left(\frac{y}{p^2};np^2\right)+\left(\frac{(-1)^{\lambda}n}{p}\right)p^{\lambda-1}c(y;n)+ p^{2\lambda-1}c\left(p^2y;\frac{n}{p^2}\right)\right) e^{2\pi i nx}.
\end{multline}
For $\wt\in \frac{1}{2}\Z$, $\mathcal{M}\in \Hcuspwt$, and a prime $p$, Theorem 7.10 of \cite{OnoUnearth} easily implies that
\begin{equation}\label{eqn:xiequiv}
\xi_{\wt}\left(\mathcal{M}\right)\big|_{\wt}T(p)=\xi_{\wt}\left(\mathcal{M}\big|_{2-\wt}T(p)\right).
\end{equation}

In addition to $\xi_{\wt}$, we require an additional operator which maps between different spaces of weak Maass forms.  The \begin{it}Maass raising operator\end{it} is defined by
\begin{equation}\label{eqn:raisingdef}
R_{\wt}:=2i\frac{\partial}{\partial \tau} +\frac{\wt}{y}.
\end{equation}
If a function $f$ satisfies weight $\wt$ automorphicity and has eigenvalue $\lambda$ under $\Delta_{\wt}$, then $R_{\wt}(f)$ satisfies weight $\wt+2$ automorphicity and has eigenvalue $\lambda+\wt$ under $\Delta_{\wt}$.  Repeated iteration is given by
$$
R_{\wt}^n:=R_{\wt+2(n-1)}\circ\cdots\circ R_{\wt+2}\circ R_{\wt}.
$$

Since the usual Petersson inner product does not always converge, we require a regularized inner product to define it for weakly holomorphic modular forms.  Hence, for $\mathcal{T}>0$ we let
$$
\FF_{\mathcal{T}}(4):= \bigcup_{j=1}^{6} \gamma_j\FF_{\mathcal{T}},
$$
where
$$
\FF_{\mathcal{T}}:=\left\{\tau\in \H\ : \ |x|\leq \frac{1}{2},\ |\tau|\geq 1,\text{ and }y\leq T\right\},
$$
and $\gamma_1,\dots,\gamma_6$ are a chosen set of coset representatives for $\SL_2(\Z)/\Gamma_0(4)$.  Following Borcherds \cite{Borcherds}, for $g\in \MM_{\wt}$ and $h\in \MM_{\wt}^!$, we define
$$
(g,h)_{\reg}:= \lim_{\mathcal{T}\to\infty}\frac{1}{6}\int_{\FF_{\mathcal{T}}(4)} g(\tau)\overline{h(\tau)} y^{\wt} \frac{dx\; dy}{y^2}.
$$
In the classical case when $g\in \SS_{k+\frac{1}{2}}$ and $h\in \MM_{k+\frac{1}{2}}$, we obtain the usual Petersson inner product $\left(g,h\right)$ and we denote the standard Petersson norm by $\|g\|$.  Using this regularized inner product, we let $\Swtperp$ be the subspace of $S_{\wt}^!$ of those weakly holomorphic modular forms with vanishing constant term which are orthogonal to cusp forms.  The corresponding subspace of $\SS_{\wt}^!$ is $\SSwtperp$.  To exhibit another way that $\DD^{1-\wt}$ is complementary to $\xi_{\wt}$, Theorem 1.2 of \cite{BruinierOnoRhoades} states that $\DD^{1-\wt}$ yields an isomorphism from $\Hcuspwt$ to $\Swtperpdual$.  Vital to the construction of Hecke eigenforms is the fact that for  $g,h\in \SS_{\wt}^!$ and $n\in \N$, the regularized Petersson inner product satisfies 
\begin{equation}\label{eqn:TnPet}
\left(g\big| T\left(n^2\right),h\right)_{\reg} = \left(g,h\big| T\left(n^2\right)\right)_{\reg}.
\end{equation}

Note that for every $\wt\geq 2$ and $n\in \N$ (resp. every discriminant $n$ for which $(-1)^{\wt-\frac{1}{2}}n<0$) there exists precisely one element of $\Swtperp$ (resp. $\SSwtperp$) of the form 
\begin{equation}\label{eqn:growth}
q^{-|n|} + O(q).
\end{equation}
Existence follows by the construction of a weight $\wt$ harmonic weak Maass Poincar\'e series which turns out to be weakly holomorphic because there are no cusp forms of weight $2-\wt$.  Uniqueness follows from the fact that the difference between any two functions satisfying \eqref{eqn:growth} vanishes at $i\infty$ and is hence a cusp form.  However, the difference is also orthogonal to cusp forms, and thus is zero.

For $\cM\in \HHcuspwt$ and $g\in \MM_{2-\wt}^!$, recall the pairing \cite{BruinierFunke}
\begin{equation}\label{eqn:pairingdef}
\left\{ g,\cM\right\}:=\left(g,\xi_{\wt}\left(\cM\right)\right)_{\reg}.
\end{equation}
In this paper, we require a generalization to all harmonic weak Maass forms.  By Proposition 3.5 of \cite{BruinierFunke}, for $g(\tau)=\sum_{n\gg -\infty} a_g(n) q^{n}\in \MM_{2-\wt}^!$ and $\cM\in \Hcuspwt$, we have 
\begin{equation}\label{eqn:pairingval}
\left\{ g,\cM\right\} = \frac{1}{6}\sum_{n\in \Z} c_{\cM}^+(n)a_g(-n).
\end{equation}

\section{The Zagier lifts}\label{sec:Zagier}
In this section, we define the Zagier lifts $\Za_D:\Hcusp\to \SSperpD{D}$ and $\Za_d:\Hcusp\to \HHcusp$.  In order to describe the constant terms in their respective Fourier expansions, for a discriminant $\Delta$ we construct the $L$-series associated to the \begin{it}Dirichlet character\end{it} $\chi_{\Delta}(\cdot):=\left(\frac{\Delta}{\cdot}\right)$ by
$$
L_{\Delta}\left(s\right):=\sum_{n=1}^{\infty}\frac{\chi_{\Delta}(n)}{n^s},
$$
which converges absolutely for $\re\left(s\right)>1$ and has a meromorphic continuation to $\C$. 

As we shall see, the other coefficients are given as traces of binary quadratic forms.  In order to state these, for each discriminant $\Delta$ we first need some auxiliary definitions involving the set $\QQ_{\Delta}$ of binary quadratic forms of discriminant $\Delta$.  The group $\SL_2(\Z)$ acts on $\QQ_{\Delta}$ in the usual way.  For discriminants $D_1$ and $D_2$, the corresponding \begin{it}genus character\end{it} (pp. 59--62 of \cite{Siegel}) of a binary quadratic form $Q\left(X,Y\right)=\left[a,b,c\right]\left(X,Y\right):=aX^2+bXY+cY^2\in \QQ_{D_1D_2}$ is given by
\begin{equation}\label{eqn:genusdef}
\chi\left(Q\right):=\begin{cases} \chi_{D_1}\left(r\right) & \text{if }\left(a,b,c, D_1\right)=1\text{ and $Q$ represents $r$ with $\left(r,D_1\right)=1$,}\\ 0 & \text{if }\left(a,b,c,D_1\right)>1.\end{cases}
\end{equation}
Whenever $D_1D_2<0$, we let $\tau_Q$ be the unique root of $Q\left(\tau_Q,1\right)=0$ in $\H$ and choose $\omega_Q=1$ unless $Q$ is $\SL_2\left(\Z\right)$-equivalent to $\left[a,0,a\right]$ or $\left[a,a,a\right]$, in which case $\omega_Q=2$ or $3$, respectively.

For an $\SL_2(\Z)$-invariant function $F:\H\to \C$ and each pair of discriminants $D_1,D_2$ with $D_1D_2<0$, the \begin{it} $\left(D_1,D_2\right)$-th twisted trace\end{it} of $F$ is given by
\begin{equation}\label{eqn:Trdef}
\Tr_{D_1,D_2}\left(F\right):= \sum_{Q\in \SL_2(\Z)\backslash\QQ_{D_1D_2}}\omega_Q^{-1}\chi(Q)F\left(\tau_Q\right).
\end{equation}
For $(-1)^kD_1>0$ and $(-1)^kD_2<0$, we generalize \cite{DukeJenkins} to define the \begin{it}modified twisted traces\end{it} $\Tr_{D_1,D_2}^*$ of $\cM\in \Hcusp$ by
\begin{equation}\label{eqn:Tr*def}
\Tr_{D_1,D_2}^*\left(\cM\right):=(-1)^{\left\lfloor \frac{k+1}{2}\right\rfloor}\left(4\pi\right)^{1-k}\left|D_1\right|^{\frac{k-1}{2}}\left|D_2\right|^{-\frac{k}{2}}\Tr_{D_1,D_2}\left( R_{2-2k}^{k-1} \left(\cM\right)\right).
\end{equation}
Note that we have only defined $\Tr_{D_1,D_2}^*$ whenever $(-1)^kD_2<0$ and $(-1)^kD_1>0$.  Reversing the order of $D_1$ and $D_2$ (with the same restrictions on the sign), we extend this definition by 
$$
\Tr_{D_2,D_1}^*(\cM):=-\Tr_{D_1,D_2}^*(\cM).
$$ 

For $\cM\in \Hcusp$ with principal part $\sum_{m< 0} c_{\cM}^+(m) q^m$, we define the \begin{it}$D$-th Zagier lift\end{it} by
\begin{equation}\label{eqn:ZaDdef}
\Za_D\left(\cM\right)\left(\tau\right):=\sum_{m>0} c_{\cM}^+\left(-m\right) m^{2k-1}\sum_{n\mid m} \chi_D(n)n^{-k}q^{-\left(\frac{m}{n}\right)^2\left|D\right|}+\sum_{\delta: \delta D<0}\Tr_{\delta,D}^*\left(\cM\right)q^{\left|\delta\right|}.
\end{equation}


We next turn to the construction of the maps $\Za_d$.  For this, we need to extend our definition to include twisted traces for indefinite binary quadratic forms.  These traces are given in terms of cycle integrals which we now describe.  For a discriminant $\Delta>0$, an indefinite binary quadratic form $Q\in \QQ_{\Delta}$, and $(t,u)$ the smallest positive solution to the Pell equation $t^2- \Delta u^2=4$,
$$
g_Q:=\left(\begin{matrix}\frac{t+bu}{2}& cu\\ -au& \frac{t-bu}{2}\end{matrix}\right)
$$
generates the (infinite cyclic) group of automorphs of $Q$.  Let $S_Q$ be the oriented semicircle given by 
\begin{equation}\label{eqn:SQdef}
a\left|\tau\right|^2 + bx+c=0
\end{equation}
directed counterclockwise if $a>0$ and clockwise if $a<0$.  For $\tau\in S_Q$, we define the directed arc $C_Q$ to be the arc from $\tau$ to $g_Q\tau$ along $S_Q$.  It was shown in Lemma 6 of \cite{DIT} that the integral along $C_Q$ of an $\SL_2(\Z)$-invariant and continuous function is both independent of $\tau\in S_Q$ and is a class invariant.  

Following Shintani \cite{Shintani} (for an explicit statement, see (6) and (7) of \cite{KohnenFourier}), for each pair $\delta,d$ the \begin{it}$(\delta,d)$-th twisted trace\end{it} ($(-1)^k\delta>0$) of $f\in S_{2k}$ is 
\begin{equation}\label{eqn:shindcycle}
\Tr_{\delta,d}\left(f\right):=-\frac{(-1)^k2^{k-2}}{3\sqrt{\pi}}\sum_{Q\in \SL_2(\Z)\backslash \QQ_{d\delta}} \chi\left(Q\right)\CC\left(f;Q\right).
\end{equation}
For each pair $\delta,d$, the \begin{it}modified $\left(\delta,d\right)$-th twisted trace\end{it} is then given by
\begin{equation}\label{eqn:shind2}
\Tr_{\delta,d}^*\left(\cM\right):=(-1)^{\left\lfloor1-\frac{k}{2}\right\rfloor}
\left(4\pi\right)^{1-k}\left|d\right|^{\frac{k-1}{2}}\left|\delta\right|^{-\frac{k}{2}}\Tr_{\delta,d}\left(\xi_{2-2k}\left(\cM\right)\right).
\end{equation}
If $\cM\in \Hcusp$ has principal part $\sum_{m< 0} c_{\cM}^+(m)q^m$ and constant coefficient $c_{\cM}^+(0)$, then the \begin{it}$d$-th Zagier lift\end{it} of $\cM$ is defined by
\begin{multline}\label{eqn:Zadpp2}
\Za_d(\cM)(\tau):=\sum_{m>0} c_{\cM}^+\left(-m\right)\sum_{n\mid m} \chi_d(n)n^{k-1}q^{-\left(\frac{m}{n}\right)^2\left|d\right|} + \frac{1}{2}L_d(1-k)c_{\cM}^+(0) + \sum_{\delta: d\delta<0}\Tr_{\delta,d}^*\left(\cM\right)q^{|\delta|}\\
+\sum_{\delta: d\delta>0} \Tr_{\delta,d}^*\left(\cM\right)\Gamma\left(k-\frac{1}{2};4\pi |\delta|y\right) q^{-|\delta|}.
\end{multline}

One can use the principal part of $\cM$ to obtain alternative definitions of $\Za_D$ and $\Za_d$.  Recall that for every $D$ and $\cM$ there exists $\widetilde{\Za}_D(\cM)\in \SSperpD{D}$ which satisfies 
\begin{equation}\label{eqn:wideZaDdef}
\widetilde{\Za}_D\left(\cM\right)\left(\tau\right)=\sum_{m>0} c_{\cM}^+\left(-m\right) m^{2k-1}\sum_{n\mid m} \chi_D(n)n^{-k}q^{-\left(\frac{m}{n}\right)^2\left|D\right|}+O(q).
\end{equation}
Moreover, the difference of two such forms is both a cusp form and orthogonal to cusp forms.  By the non-degeneracy of the Petersson inner product, for each pair $D,\cM$ the form satisfying \eqref{eqn:wideZaDdef} is hence unique.  To obtain a similar definition for $\Za_d$, recall that for each pair $d,\cM$ there is a unique element $\widetilde{\Za}_d(\cM)\in\HHcusp$ of the form
\begin{equation}\label{eqn:Zadpp}
\widetilde{\Za}_d(\cM)(\tau)=\sum_{m>0} c_{\cM}^+\left(-m\right)\sum_{n\mid m} \chi_d(n)n^{k-1}q^{-\left(\frac{m}{n}\right)^2\left|d\right|} +O(1).
\end{equation}
\begin{remark}
The existence and uniqueness of the form satisfying the right hand side of  \eqref{eqn:Zadpp} follows from the theory of Poincar\'e series, which we recall in Section \ref{sec:Poincare}.
\end{remark}

A comparison of two different ways to compute the coefficients of this unique harmonic weak Maass form leads to the unexpected connection in Theorem \ref{thm:Zaddef2} between the twisted traces defined by the classical cycle integrals given in \eqref{eqn:shindcycle} and cycle integrals of weak Maass forms.  The latter are given by
$$
\widetilde{\Tr}_{\delta,d}^*\left(\cM\right):=(-1)^{\left\lfloor 1 -\frac{k}{2}\right\rfloor}
\left(4\pi\right)^{1-k}\left|d\right|^{\frac{k-1}{2}}\left|\delta\right|^{-\frac{k}{2}}\widetilde{\Tr}_{\delta,d}\left(R_{2-2k}^{k-1} \left(\cM\right)\right),
$$
where (in the case of convergence) for an $\SL_2(\Z)$-invariant function $F$,
\begin{equation}\label{eqn:wideTr}
\widetilde{\Tr}_{\delta,d}(F):=\frac{\Gamma\left(\frac{k+1}{2}\right)}{2\sqrt{\pi}\Gamma\left(\frac{k}{2}\right)}\sum_{Q\in \SL_2(\Z)\backslash\QQ_{\delta d}} \chi\left(Q\right)\CC\left(F;Q\right).
\end{equation}
\begin{remark}
In order to uniformly consider traces of positive and negative discriminant, whenever $D_1D_2<0$ we also denote 
$$
\widetilde{\Tr}_{D_1,D_2}(F):=\Tr_{D_1,D_2}\left(F\right),
$$
where $\Tr_{D_1,D_2}\left(F\right)$ was defined in \eqref{eqn:Trdef}.
\end{remark}
Theorem \ref{thm:Zaddef2} is equivalent to the next theorem, whose proof is deferred to Proposition \ref{prop:Zmodular}.
\begin{theorem}\label{thm:Zaddef}
If $\delta d>0$ is not a square and not both $\delta$ and $d$ are negative, then one has 
$$
c_{\Za_d(\cM)}^-\left(-\left|\delta\right|\right)= \widetilde{\Tr}_{\delta,d}^*\left(\cM\right).
$$
\end{theorem}

\section{Poincar\'e series}\label{sec:Poincare}
In this section, we describe an important family of Poincar\'e series with $\Gamma$ chosen as in \eqref{eqn:Gammadef}.  For an integer $m\neq 0$ and a function $\varphi:\R^{+}\to \C$ satisfying $\varphi(y)=O\left(y^\alpha\right)$ ($y\to 0$) for some $\alpha\in \R$, we set
$$
\varphi^{*}_m(\tau):=\varphi(y)e^{2\pi i mx}.
$$
Such functions are fixed by the translations $\Gamma_\infty:=\left\{ \pm\left( \begin{smallmatrix} 1&n\\0&1\end{smallmatrix}\right) :\ n\in \Z\right\}$.  
Given  this data, define the generic \begin{it}Poincar\'e series\end{it}
\begin{equation}\label{eqn:Poincgen}
\PP(m,\wt,\varphi;\tau):= \sum_{\gamma\in\Gamma_\infty \backslash\Gamma}\varphi^{*}_m \big|_{\wt} \gamma(\tau),
\end{equation}
which converge absolutely and uniformly for $\wt>2-2\alpha$.  

To define the Maass--Poincar\'e series, one specifically chooses
\begin{equation}\label{eqn:psidef}
\psi_{m,\wt}\left(s;y\right):=\left(4\pi \left|m\right|\right)^{\frac{\wt}{2}}\Gamma\left(2s\right)^{-1}\mathcal{M}_{\wt,s}\left(4\pi my\right).
\end{equation}
Here for complex $s$ and $u\neq 0$
\begin{equation*}
\mathcal{M}_{\wt,s}(u):= |u|^{-\frac{\wt}{2}}
M_{\frac{\wt}{2}\sgn(u),\,s-\frac{1}{2}}\left(|u|\right),
\end{equation*}
where   $M_{\nu,\,\mu}(u)$ is the usual $M$-Whittaker function.  In the special cases that $s=\frac{\wt}{2}$ or $s=1-\frac{\wt}{2}$ the resulting functions $\psi_{m,\wt}^*(s;\tau)$ are harmonic.  Specifically, for $u>0$ 
\begin{align*}
\mathcal{M}_{\wt,\frac{\wt}{2}}(-u)&= e^{\frac{u}{2}},\\
\mathcal{M}_{\wt,1-\frac{\wt}{2}}(-u)&=(\wt-1)e^{\frac{u}{2}}\Gamma\left(1-\wt;u\right)+(1-\wt)\Gamma(1-\wt)e^{\frac{u}{2}}.
\end{align*}
Using the notation in \eqref{eqn:Poincgen}, the family of Poincar\'e series which we require is defined by
\begin{equation}\label{eqn:Poincs}
P_{m,\wt}\left(s;\tau\right):=\begin{cases}
\PP\left(m,\wt,\psi_{m,\wt}\left(s;y\right);\tau\right) &\text{if }\wt\in \Z,\\
\frac{3}{2}\PP\left(m,\wt,\psi_{m,\wt}\left(s;y\right);\tau\right)\big|\pr &\text{if }\wt\in \frac{1}{2}\Z\setminus \Z.
\end{cases}
\end{equation}
These Poincar\'e series converge absolutely and uniformly whenever $\re\left(s\right)>1$.  For $\wt\geq 2$, the classical family \cite{IwaniecAut} of weakly holomorphic modular forms is defined by
\begin{equation}\label{eqn:PmClassical}
P_{m,\wt}\left(\tau\right):= \left(4\pi \left|m\right|\right)^{-\frac{\wt}{2}}\Gamma\left(1+\left|\wt-1\right|\right)P_{m,\wt}\left(\frac{\wt}{2};\tau\right).
\end{equation}
Similarly, the Poincar\'e series
\begin{equation}\label{eqn:Fmdef}
F_{m,2-\wt}\left(\tau\right):= \left(4\pi \left|m\right|\right)^{\frac{\wt}{2}-1}P_{-m,2-\wt}\left(\frac{\wt}{2};\tau\right)
\end{equation}
 are harmonic weak Maass forms (for example, see \cite{Fay} or \eqref{eqn:Deltaact} below) with principal part $q^{-m}$ whenever $m>0$.  
For notational consistency, we also denote the positive weight weakly holomorphic form by 
$$
F_{m,\wt}(\tau):=\left(4\pi \left|m\right|\right)^{\frac{\wt}{2}-1}P_{-m,\wt}\left(\frac{\wt}{2};\tau\right).
$$

In the following lemma, we collect the images of $P_{m,\wt}
\left(s;\tau\right)$ under the operators which are relevant for Maass forms.  
\begin{lemma}\label{lem:Poincops}
For $m\in \Z\setminus\{0\}$, $\wt\in \frac{1}{2}\Z$, and $s\in \C$ with $\re\left(s\right)>1$, we have 
\begin{eqnarray}
\label{eqn:Deltaact}
\Delta_{\wt}\left(P_{m,\wt}\left(s;\tau\right)\right) &=& \left(s-\frac{\wt}{2}\right)\left(1-s-\frac{\wt}{2}\right)P_{m,\wt}\left(s;\tau\right),\\
\label{eqn:xiact}
\xi_{\wt}\left(P_{m,\wt}\left(s;\tau\right)\right) &=& \left(\overline{s}-\frac{\wt}{2}\right)P_{-m,2-\wt}\left(\overline{s};\tau\right),\\
\label{eqn:lowerP}
L_{\wt}\left(P_{m,\wt}\left(s;\tau\right)\right) &=& \left(s-\frac{\wt}{2}\right)P_{m,\wt-2}\left(s;\tau\right),\\
\label{eqn:raiseP}
R_{\wt}\left(P_{m,\wt}\left(s;\tau\right)\right) &=& \left(s+\frac{\wt}{2}\right) P_{m,\wt+2}\left(s;\tau\right),
\end{eqnarray}
Moreover, if $2-\wt\in 2\N$, then 
\begin{eqnarray}
\label{eqn:ZadPoinc}
\widetilde{\Za}_d\left(F_{m,\wt}\right) &=&\sum_{n\mid m} \chi_d(n)n^{
-\frac{\wt}{2}}F_{\left(\frac{m}{n}\right)^{2}\left|d\right|,\frac{\wt+1}{2}},\\
\label{eqn:ZaDPoinc}
\widetilde{\Za}_D\left(F_{m,\wt}\right) &=& |m|^{1-\wt}\sum_{n\mid m} |n|^{\frac{\wt}{2}-1}F_{\left(\frac{m}{n}\right)^{2}\left|D\right|,\frac{3-\wt}{2}},\\
\label{eqn:DerP}
\Der\left(F_{|d|,\frac{\wt+1}{2}}\right)&=& P_{-|D|,\frac{3-\wt}{2}}.
\end{eqnarray}

\end{lemma}
\begin{proof}
Except for the last three identities, the lemma follows by a standard calculation involving the following relations between Whittaker functions:
\begin{eqnarray*}
M_{\ell,s-\frac{1}{2}}(y)&=& e^{\pi i \ell}\frac{\Gamma(2s)}{\Gamma(s-\ell)}W_{-\ell,s-\frac{1}{2}}(-y) + e^{\pi i\left(\ell-s\right)}\frac{\Gamma(2s)}{\Gamma(s+\ell)}W_{\ell,s-\frac{1}{2}}(y),\\
W_{\ell,m}(y)&=& y^{\frac{1}{2}}W_{\ell-\frac{1}{2},m-\frac{1}{2}}(y) + \left(\frac{1}{2}-\ell+m\right)W_{\ell-1,m}(y),\\
W_{\ell,m}(y)&=& y^{\frac{1}{2}}W_{\ell-\frac{1}{2},m+\frac{1}{2}}(y) + \left(\frac{1}{2}-\ell-m\right)W_{\ell-1,m}(y),\\
y W_{\ell,m}'(y) &=& \left(\ell-\frac{y}{2}\right)W_{\ell,m}(y)-\left(m^2-\left(\ell-\frac{1}{2}\right)^2\right)W_{\ell-1,m}(y).
\end{eqnarray*}

By comparing principal parts, uniqueness implies \eqref{eqn:ZadPoinc}.  If the right hand side of \eqref{eqn:ZaDPoinc} were an element of $\SSperp$, then uniqueness would again imply \eqref{eqn:ZaDPoinc} by comparing the principal parts.  This turns out to be the case, but we defer the proof of its orthogonality to Proposition \ref{prop:orthog}.  The last identity shall be proven in \eqref{eqn:DerPoincmap}, but we include it here for completeness.
\end{proof}
For $\wt\in\frac{1}{2}\Z\setminus\Z$ and a discriminant $m$ satisfying $\sgn(m)=\sgn(\wt-1)$, we also require the Fourier expansion of $P_{m,\wt}\left(s;\tau\right)$ for the modularity of the lifts.  For this, we define the half-integral weight Kloosterman sums
$$
K_{\wt}\left(m,n;c\right):=2^{-\frac{1}{2}}\left(1+\left(\frac{4}{c}\right)\right)\left(1-\left(-1\right)^{\wt-\frac{1}{2}}i\right)\sum_{\nu\pmod{4c}^*} \left(\frac{4c}{\nu}\right)\varepsilon_{\nu}^{2\wt} e^{2\pi i \left(\frac{m\overline{\nu}+n\nu}{4c}\right)}. 
$$
Here the sum runs over  $\nu\pmod{4c}$ relatively prime to $4c$ and $\overline{\nu}$ denotes the inverse of $\nu \pmod{4c}$.  We also require the modified $W$-Whittaker function
\begin{equation}\label{eqn:Wdef}
\W_{n,\wt}\left(s;y\right):=\begin{cases} 
\Gamma\left(s+\sgn(n)\frac{\wt}{2}\right)^{-1} y^{-\frac{\wt}{2}} W_{\frac{\wt}{2}\sgn(n), s-\frac{1}{2}}\left(4\pi |n|y\right) &\text{if }n\neq 0,\\
\frac{y^{1-\frac{\wt}{2}-s}}{\left(2s-1\right)\Gamma\left(s-\frac{\wt}{2}\right)\Gamma\left(s+\frac{\wt}{2}\right)} &\text{if }n=0.\end{cases}
\end{equation}
The following special values of $\W_{n,\wt}\left(s;y\right)$ for $n\neq 0$ prove useful:
\begin{equation}\label{eqn:Wspecial}
\W_{n,\wt}\left(1-\frac{\wt}{2};y\right) = 
\begin{cases}
(4\pi n)^{\frac{\wt}{2}}e^{-2\pi n y}&\text{if }n>0,\\
\frac{\left(4\pi|n|\right)^{\frac{\wt}{2}}}{\Gamma\left(1-\wt\right)} e^{2\pi \left|n\right| y}\Gamma\left(1-\wt;4\pi \left|n\right|y\right) &\text{if }n<0.
\end{cases}
\end{equation}
Furthermore, 
\begin{equation}\label{eqn:Wspecial2}
\W_{n,\wt}\left(\frac{\wt}{2};y\right)=
\begin{cases}
\left(4\pi n\right)^{\frac{\wt}{2}}e^{-2\pi ny} & \text{if } n>0,\\
0 & \text{if } n\leq 0.
\end{cases}
\end{equation}
Calculations similar to those necessary to obtain the following lemma have been carried out in a number of places (for example, see \cite{Fay}).
\begin{lemma}\label{lem:Poinccoeff}
For $\wt\in \frac{1}{2}\Z\setminus\Z$, $\re(s)>1$, and a discriminant $m$ satisfying $\sgn(m)=\sgn(\wt-1)$, $P_{m,\wt}\left(\tau,s\right)$ has the Fourier expansion
\begin{equation}\label{eqn:PmFourier}
P_{m,\wt}\left(s;\tau\right)= \psi_{m,\wt}\left(s;y\right)e^{2\pi imx} + \sum_{n\equiv 0,1\pmod{4}}
b_{m,\wt}\left(s;n\right)\W_{n,\wt}\left(s;y\right)e^{2\pi i nx},
\end{equation}
where
\begin{equation}\label{eqn:bcoeff}
b_{m,\wt}\left(s;n\right):= 2\pi (-1)^{\left\lfloor\frac{2\wt+1}{4}\right\rfloor}\sum_{c>0} K_{\wt}\left(m,n;c\right)
\begin{cases} 
 \left|\frac{m}{n}\right|^{\frac{1}{2}}\cdot \frac{1}{4c} I_{2s-1}\left(\frac{\pi\sqrt{|mn|}}{c}\right)&\text{if }mn<0,\\
  \left(\frac{m}{n}\right)^{\frac{1}{2}}\cdot  \frac{1}{4c} J_{2s-1}\left(\frac{\pi\sqrt{mn}}{c}\right)&\text{if }mn>0,\\
2\frac{\pi^{s}\left|m\right|^s}{(4c)^{2s}}&\text{if }n=0.
\end{cases}
\end{equation}
Here $I$ and $J$ denote the $I$-Bessel and $J$-Bessel functions.
\end{lemma}

Recall further that the coefficients at the cusps $0$ and $\frac{1}{2}$ can be written in terms of the coefficients at $i\infty$.  In particular, for a translation invariant function
$$
G(\tau):=\sum_{n\in \Z} a(n;y) e^{2\pi i nx}
$$
define
\begin{align*}
G^e(\tau)&:=\sum_{n\equiv 0\pmod{2}} a\left(n;\frac{y}{4}\right) e^{\frac{2\pi i n x}{4}},\\
G^o(\tau)&:=\sum_{n\equiv 1\pmod{2}} a\left(n;\frac{y}{4}\right)e^{\frac{2\pi i n}{8}} e^{\frac{2\pi i n x}{4}}.
\end{align*}
Then one has (see \cite{Kohnen} for the holomorphic case)
\begin{align}\label{eqn:cuspeven}
P_{m,\wt}\left(s;-\frac{1}{4\tau}\right) &= 2^{\frac{1}{2}-\wt}(-1)^{\left\lfloor \frac{2\wt+1}{4}\right\rfloor}\left(\frac{2\tau}{i}\right)^{\wt} P_{m,\wt}^{e}\left(s;\tau\right),\\
\label{eqn:cuspodd}
P_{m,\wt}\left(s;\frac{\tau}{2\tau+1}\right)&= 2^{\frac{1}{2}-\wt}(-1)^{\left\lfloor \frac{2\wt+1}{4}\right\rfloor}\left(\frac{2\tau+1}{i}\right)^{\wt}P_{m,\wt}^{o}(s;\tau).
\end{align}

For $m\in \Z\setminus\{0\}$, $\re\left(s\right)>1$, we define the \begin{it}Niebur Poincar\'e series\end{it} by
\begin{equation}\label{eqn:Ndef}
F_m\left(s;\tau\right):=\PP\left(m,0,\phi_{m,s};\tau\right)
\end{equation}
where 
$$
\phi_{m,s}\left(y\right):=2\pi \left|m\right|^{s-\frac{1}{2}} y^{\frac{1}{2}}I_{s-\frac{1}{2}}\left(2\pi \left|m\right|y\right).
$$
The functions $F_m$ are weight $0$ weak Maass forms with eigenvalue $s\left(1-s\right)$ under $\Delta=\Delta_0$ \cite{Niebur}.
 
For $\re\left(s\right)>1$, the Niebur Poincar\'e series are related to the family $P_{m,\wt}\left(s;\tau\right)$ by the identity (see (13.1.32) and (13.6.3) of \cite{AS})
$$
M_{0,s-\frac{1}{2}}\left(2y\right) = 2^{2s-\frac{1}{2}} \Gamma\left(s+\frac{1}{2}\right)y^{\frac{1}{2}}I_{s-\frac{1}{2}}\left(y\right). 
$$
We rewrite $y^{\frac{1}{2}}I_{s-\frac{1}{2}}\left(2\pi \left|m\right|y\right)$ in terms of $\psi_{m,0}\left(s;y\right)$ via \eqref{eqn:psidef} and use the duplication formula for the $\Gamma$-function to obtain
\begin{equation}\label{eqn:Niebur}
P_{m,0}\left(s;\tau\right) =\Gamma\left(s\right)^{-1}\left|m\right|^{1-s}F_{m}\left(s;\tau\right).
\end{equation}

\section{The Zagier lift and the Shintani lift}\label{sec:ZagShin}
In this section we prove a theorem relating the lifts $\widetilde{\Za}_d$ (defined in \eqref{eqn:Zadpp}) to the (classical) Shintani lifts.  It is useful to first show that (the alternative versions of) the Zagier lifts commute with the Hecke operators.  
\begin{lemma}\label{lem:Hecke}
If $\cM\in \Hcusp$, then
\begin{eqnarray}\label{eqn:ZagierHeckeD}
\widetilde{\Za}_D\left(\cM\right)\big|_{k+\frac{1}{2}} T\left(n^2\right) &=& \widetilde{\Za}_D\left(\cM\big|_{2-2k}T(n)\right),\\
\label{eqn:ZagierHecked}
\widetilde{\Za}_d\left(\cM\right)\big|_{\frac{3}{2}-k} T\left(n^2\right) &=& \widetilde{\Za}_d\left(\cM\big|_{2-2k}T(n)\right).
\end{eqnarray}
\end{lemma}

\begin{proof}
Since $\widetilde{\Za}_D\left(\cM\right)$ is orthogonal to cusp forms, it is uniquely determined by its principal parts.  Moreover, $\SSperp$ is preserved under the action of the Hecke operators by \eqref{eqn:TnPet}.  It hence suffices to show that the principal parts of the two sides of \eqref{eqn:ZagierHeckeD} and \eqref{eqn:ZagierHecked} match.  This calculation was carried out for $\cM\in M_{2-2k}^!$ \cite{DukeJenkins} and the proof follows mutatis mutandis for $\cM\in \Hcusp$.  
For \eqref{eqn:ZagierHecked}, it suffices to compare the principal parts of both sides, since $\frac{3}{2}-k<0$.  Since this is a formal calculation which is analogous to the computation given in \cite{DukeJenkins} we skip it here.
\end{proof}

The next theorem is one of the main steps in the proof of Theorem \ref{thm:Zaddef2}.
\begin{theorem}\label{thm:tildeprop}
For each $\cM\in \Hcusp$, we have
\begin{equation}\label{eqn:tildeprop}
\xi_{\frac{3}{2}-k}\left(\widetilde{\Za}_d(\cM)\right)=\frac{1}{3}(-1)^{\left\lfloor \frac{k}{2}\right\rfloor}2^{k-1}\shin_d\left(\xi_{2-2k}(\cM)\right).
\end{equation}
\end{theorem}
\begin{proof}
By comparing principal parts, it is easy to show that 
\begin{equation}\label{eqn:FTn}
\left\{F_{1,2-2k}\big|_{2-2k}T(n)\Big|n\in \N\right\}
\end{equation}
spans $\Hcusp$.  Hence it suffices to prove \eqref{eqn:tildeprop} for elements of \eqref{eqn:FTn}.  We now use the Hecke equivariance of the relevant operators to prove that the $n=1$ case suffices.  The proof follows by induction on the number of divisors of $n$.  To see the induction step, assume that \eqref{eqn:tildeprop} holds for some $\cM\in \Hcusp$.  Recall that for $f\in S_{2k}$, one has the classical Hecke equivariance
\begin{equation}\label{eqn:shinequiv}
\shin_d\left(f\big|_{2k}T(p)\right) = \shin_d(f)\big|_{k+\frac{1}{2}}T\left(p^2\right).
\end{equation}
Then we may combine \eqref{eqn:shinequiv} with \eqref{eqn:xiequiv} and \eqref{eqn:ZagierHecked} to yield
\begin{align*}
\xi_{\frac{3}{2}-k}\left(\widetilde{\Za}_d\left(\cM\big|_{2-2k}T(p)\right)\right)&= \xi_{\frac{3}{2}-k}\left(\widetilde{\Za}_d(\cM)\right)\big|_{\frac{3}{2}-k}T\left(p^2\right)\\
&= \frac{1}{3}(-1)^{\left\lfloor \frac{k}{2}\right\rfloor}2^{k-1}\shin_d\left(\xi_{2-2k}(\cM)\right)\big|_{\frac{3}{2}-k}T\left(p^2\right)\\
&=\frac{1}{3}(-1)^{\left\lfloor \frac{k}{2}\right\rfloor}2^{k-1}\shin_d\left(\xi_{2-2k}\left(\cM\big|_{2-2k}T(p)\right)\right).
\end{align*}

We have hence reduced the problem to showing \eqref{eqn:tildeprop} for $\cM=F_{1,2-2k}$.  Suppose that $\left\{f_1,\dots, f_{\ell}\right\}$ $\left(\ell:=\dim\left(S_{2k}\right)\right)$ is a basis of Hecke eigenforms for $S_{2k}$ and $h_j(\tau):=\sum_{n=1}^{\infty} c_j\left(n\right)q^n\in \SS_{k+\frac{1}{2}}$ is a Hecke eigenform with the same eigenvalues as $f_j$ \cite{Kohnen}.  
We use the pairing \eqref{eqn:pairingdef}, given by 
$$
\left\{ h_j, \widetilde{\Za}_d\left(F_{1,2-2k}\right)\right\} = \left( h_j, \xi_{\frac{3}{2}-k}\left(\widetilde{\Za}_d\left(F_{1,2-2k}\right)\right)\right).
$$
By \eqref{eqn:ZadPoinc}, we have
$$
\widetilde{\Za}_d\left(F_{1,2-2k}\right) = F_{\left|d\right|,\frac{3}{2}-k},
$$
which has principal part $q^{-|d|}$.  Hence by \eqref{eqn:pairingval} and the fact that $h_j$ is a cusp form,
$$
\left(\xi_{\frac{3}{2}-k}\left(\widetilde{\Za}_d\left(F_{1,2-2k}\right)\right),h_j\right)=\overline{\left\{ h_j, F_{\left|d\right|,\frac{3}{2}-k}\right\}} = \frac{1}{6} \overline{c_{j}\left(|d|\right)}.
$$
Therefore, we have
\begin{equation}\label{eqn:shindHval}
\xi_{\frac{3}{2}-k}\left(\widetilde{\Za}_d\left(F_{1,2-2k}\right)\right) = \frac{1}{6} \sum_{j=1}^{\ell} \overline{c_{j}\left(|d|\right)} \frac{h_j}{\| h_j\|^2}.
\end{equation}

To compute the right hand side of \eqref{eqn:tildeprop}, we use \eqref{eqn:PmClassical}, \eqref{eqn:Fmdef}, and \eqref{eqn:xiact}, to first compute
\begin{equation}\label{eqn:xiFmval}
\xi_{2-2k}\left(F_{1,2-2k}(\tau)\right)=(4\pi)^{-k}\left(2k-1\right)P_{1,2k}\left(k;\tau\right)= \frac{\left(4\pi\right)^{2k-1}}{\left(2k-2\right)!} P_{1,2k}\left(\tau\right).
\end{equation}
We next rewrite $P_{1,2k}$ in terms of the Hecke eigenforms $f_j$.  Using the Petersson coefficient formula (cf. Theorem 3.3 of \cite{IwaniecAut}), we have that
$$
\left(f_j,P_{1,2k}\right) = \frac{\left(2k-2\right)!}{\left(4\pi\right)^{2k-1}}a_j(1),
$$
where $a_j(n)$ denotes the $n$th Fourier coefficient of $f_j$.  
Since $a_j(1)=1$ for every $j$, we therefore have
\begin{equation}\label{eqn:P1decompose}
P_{1,2k}= \frac{\left(2k-2\right)!}{\left(4\pi \right)^{2k-1}}\sum_{j=1}^{\ell}\frac{f_j}{\|f_j\|^2}
\end{equation}
By linearity, we only need to compute the image of each $f_j$ under $\shin_d$.  For this 
we use \eqref{eqn:shindef}
 and Theorem 3 of \cite{KohnenFourier} to evaluate
$$
\shin_{d}\left(f_j\right)=\left(-1\right)^{\left\lfloor \frac{k}{2}\right\rfloor}2^{-k}\|f_j\|^2 \overline{c_j(|d|)}\frac{h_j}{\|h_j\|^2}.
$$
Plugging this into \eqref{eqn:P1decompose} and using \eqref{eqn:xiFmval}
yields 
$$
\shin_d\left(\xi_{2-2k}\left(F_{1,2-2k}\right)\right)= \frac{\left(4\pi \right)^{2k-1}}{\left(2k-2\right)!}\shin_d\left(P_{1,2k}\right) = \left(-1\right)^{\left\lfloor \frac{k}{2}\right\rfloor}2^{-k} \sum_{j=1}^{\ell}\overline{c_j(|d|)}\frac{h_j}{\|h_j\|^2}.
$$
This completes the proof.

\end{proof}

\section{Proof of Theorems \ref{thm:Zaddef2} and \ref{thm:Zaharmonic}}\label{sec:ortho}
In this section, we prove Theorem \ref{thm:Zaharmonic} and Theorem \ref{thm:Zaddef2} (which we have rewritten as the equivalent statement in Theorem \ref{thm:Zaddef}).  We separate Theorem \ref{thm:Zaharmonic} into 3 propositions, one establishing the automorphicity of $\Za_D$, one yielding the orthogonality, and finally one showing that they are both bijections on the relevant spaces.  

\subsection{The constant term}
Before showing the modularity of the Zagier lifts, for $\cM\in \Hcusp$, we prove a useful lemma giving the constant term of $\widetilde{\Za}_d(\cM)$, defined in \eqref{eqn:Zadpp}.
\begin{lemma}\label{lem:const}
If $\cM\in \Hcusp$ has constant term $c_{\cM}^+(0)$ in the expansion \eqref{eqn:fourier}, then the constant term of $\widetilde{\Za}_d(\cM)$ is
\begin{equation}\label{eqn:const}
c_{\widetilde{\Za}_d(\cM)}^+(0) = \frac{1}{2}L_{d}(1-k)c_{\cM}^+(0).
\end{equation}
\end{lemma}
\begin{proof}
Since the Poincar\'e series $F_{m,2-2k}$ with $m\in \N$ span the space $\Hcusp$, it suffices to show \eqref{eqn:const} for $\cM=F_{m,2-2k}$.  We first use \eqref{eqn:ZadPoinc} to write
\begin{equation}\label{eqn:Mdexp}
\widetilde{\Za}_d\left(F_{m,2-2k}\right) = \cM_{d}:=\sum_{n\mid m} \chi_d(n)n^{k-1}\cM_{d,n},
\end{equation}
where 
\begin{equation}\label{eqn:cMdn}
\cM_{d,n}:=F_{\left(\frac{m}{n}\right)^{2}\left|d\right|,\frac{3}{2}-k}.
\end{equation}
We next require an extension of Zagier's duality \eqref{eqn:Zagduality} to the case $D=0$.  Although well-known to experts, we provide a proof here for the convenience of the reader.  We let $\GG_{\delta}\in \HHcusp$ ($\delta$ a discriminant) be the unique element satisfying $\GG_{\delta}(\tau)=q^{-|\delta|}+O(1)$ and denote by $\Eis\in\MM_{k+\frac{1}{2}}$ the weight $k+\frac{1}{2}$ Eisenstein series, normalized to have constant coefficient $1$.  We use the pairing \eqref{eqn:pairingdef} and compute
$$
\left\{ \Eis,\GG_{\delta} \right\} = \left(\Eis ,\xi_{\frac{3}{2}-k}\left(\GG_{\delta}\right)\right)=0,
$$
because $\Eis$ is orthogonal to cusp forms.  Denoting the $
|\delta|$-th coefficient of $\Eis$ by $c(0,\delta)$ and the $0$-th coefficient of $\GG_{\delta}$ by $a(\delta,0)$, we use \eqref{eqn:pairingval} to evaluate
$$
0=\left\{ \Eis,\GG_{\delta} \right\}=\frac{1}{6}\left(a(\delta,0)+c(0,\delta)\right).
$$
This establishes \eqref{eqn:Zagduality} for $D=0$.  

To compute the coefficients of $\Eis$, recall that the $\left(\frac{m}{n}\right)^2|d|$-th coefficient of Cohen's Eisenstein series \cite{Cohen} is given by 
$$
L_d(1-k)\sum_{r\mid \frac{m}{n}}\mu(r)\chi_{d}(r)r^{k-1}\sigma_{2k-1}\left(\frac{m}{rn}\right)
$$
while its constant coefficient equals $\zeta(1-2k)$.  Applying duality and M\"obius inversion hence yields
\begin{equation}\label{eqn:Zadconst}
c_{\cM_d}^+(0) = -\frac{L_{d}(1-k)}{\zeta(1-2k)}\sum_{n\mid m} \chi_d(n)n^{k-1}\sum_{r\mid \frac{m}{n}}\mu(r)\chi_{d}(r)r^{k-1}\sigma_{2k-1}\left(\frac{m}{rn}\right)=-\frac{L_{d}(1-k)}{\zeta(1-2k)}\sigma_{2k-1}(m).
\end{equation}
Again applying (an integral weight version of) equation \eqref{eqn:Zagduality} with $D=0$ (the proof is exactly as above), the $m$-th coefficient of the normalized weight $2k$ Eisenstein series, namely
$$
-\frac{4k}{B_{2k}}\sigma_{2k-1}(m)=\frac{2}{\zeta(1-2k)}\sigma_{2k-1}(m),
$$
is the negative of the constant coefficient of $F_{m,2-2k}$.  Thus \eqref{eqn:Zadconst} becomes
$$
c_{\cM_d}^+(0)=\frac{c_{\cM}^+(0)}{2}L_d(1-k),
$$
as desired.
\end{proof}

\subsection{Modularity}
The next step in the proof of Theorem \ref{thm:Zaharmonic} is to show the automorphicity of our lifts.  As a side effect, we obtain Theorem \ref{thm:Zaddef} and hence establish Theorem \ref{thm:Zaddef2}.
\begin{proposition}\label{prop:Zmodular}
If $\cM\in \Hcusp$, then
\begin{eqnarray}
\nonumber&\Za_D\left(\cM\right)\in \MM_{k+\frac{1}{2},D}^!,&\\
\label{eqn:Zadeq}
&\Za_d\left(\cM\right)=\widetilde{\Za}_d\left(\cM\right)\in \HHcusp.&
\end{eqnarray}
Furthermore, the functions $\Za_D(\cM)$ are contained in the space spanned by the Poincar\'e series $P_{(-1)^kDm^2,k+\frac{1}{2}}$ ($m\in \N$) and every $\Za_d\left(\cM\right)$ satisfies Theorem \ref{thm:Zaddef}.
\end{proposition}
\begin{proof}
Since the Poincar\'e series $F_{m,2-2k}$ with $m\in \N$ form a basis for $\Hcusp$, it suffices to work on the level of Poincar\'e series and determine each of their images under $\Za_d$ and $\Za_D$.  As stated in the proof of Proposition 5 of \cite{DIT}, for $m\neq 0$ and discriminants $d_1,d_2$ which are not both negative and for which  $d_1d_2$ is not a square, the traces of the Niebur Poincar\'e series, defined in \eqref{eqn:Ndef}, are given by
\begin{equation}\label{eqn:Trcoeff}
\widetilde{\Tr}_{d_1,d_2}\left(F_m\left(s;\tau\right)\right)=\sum_{n\mid m} \left(\frac{d_2}{n}\right) a_{n}\left(s\right),
\end{equation}
where (note the different normalization)
\begin{equation}\label{eqn:anval}
a_n\left(s\right):=2\pi\left|m\right|^{s-\frac{1}{2}}\left|d_1d_2\right|^{\frac{1}{4}}  \left|n\right|^{-\frac{1}{2}}\sum_{c>0} \frac{K_{\frac{1}{2}}\left(d_1, \left(\frac{m}{n}\right)^2 d_2;c\right)}{4c} 
\begin{cases} 
I_{s-\frac{1}{2}}\left(\frac{\pi}{c}\left|\frac{m}{n}\right|\sqrt{\left|d_1d_2\right|} \right)& \text{if }d_1d_2<0,\\
J_{s-\frac{1}{2}}\left(\frac{\pi}{c}\left|\frac{m}{n}\right|\sqrt{d_1d_2}\right)& \text{if }d_1d_2>0.
\end{cases} 
\end{equation}

We use $\widetilde{\Za}_d$ defined in \eqref{eqn:Zadpp} and show that the holomorphic part of $\widetilde{\Za}_d\left(F_{m,2-2k}\right)$ matches the holomorphic part of $\Za_d\left(F_{m,2-2k}\right)$ for every $m\in \N$.  We simultaneously prove Theorem \ref{thm:Zaddef}, since these cases may be treated uniformly.

The constant coefficients of $\widetilde{\Za}_{d}$ and $\Za_d$ match by \eqref{eqn:const}.  In order to use \eqref{eqn:Trcoeff} to compare the other coefficients of the right hand side of \eqref{eqn:ZadPoinc} with the coefficients given in Theorem \ref{thm:Zaddef}, we relate $R_{2-2k}^{k-1}\left(F_{m,2-2k}\right)\left(\tau\right)$ to $F_m\left(k;\tau\right)$.  By \eqref{eqn:Fmdef}, \eqref{eqn:raiseP}, and \eqref{eqn:Niebur}, we obtain
\begin{equation}\label{eqn:Rlift}
R_{2-2k}^{k-1}\left( F_{m,2-2k}\left(\tau\right)\right)=\left(4\pi m\right)^{k-1}(k-1)!P_{-m,0}\left(k;\tau\right) = \left(4\pi\right)^{k-1} F_{-m}\left(k;\tau\right).
\end{equation}
We now choose $d_2=d$ and $s=k$ in \eqref{eqn:Trcoeff} and relate $a_n(k)$ to the coefficient $c_{\cM_{d,n}}^{\varepsilon}\left(\varepsilon\left|d_1\right|\right)$ in \eqref{eqn:fourier}, where $\varepsilon:=\sgn\left((-1)^{k+1}d_1\right)$.  For this, we first rewrite $a_n(k)$ in terms of  
$$
b_n:=b_{-\left(\frac{m}{n}\right)^2\left|d\right|,\frac{3}{2}-k}\left(\frac{k}{2}+\frac{1}{4};\varepsilon\left|d_1\right|\right),
$$
which were defined in \eqref{eqn:bcoeff}.

Recall that $K_{\wt}\left(m,n;c\right)=K_{\wt}\left(n,m;c\right)$, $K_{\wt+2}\left(m,n;c\right)=K_{\wt}\left(m,n;c\right)$, and $K_{2-\wt}\left(-n,-m;c\right)= K_{\wt}\left(m,n;c\right)$ (e.g., see Proposition 3.1 of \cite{BringmannOnoMathAnn}).  Hence 
\begin{equation}\label{eqn:Kloosterman}
K_{\frac{1}{2}}\left(d_1, \left(\frac{m}{n}\right)^2 d;c\right) = K_{\frac{3}{2}-k}\left(-\left(\frac{m}{n}\right)^2 \left|d\right|,\varepsilon\left|d_1\right|;c\right).
\end{equation}
Thus \eqref{eqn:bcoeff} implies that
$$
b_n =(-1)^{\left\lfloor 1-\frac{k}{2}\right\rfloor}2\pi  \left|\frac{m}{n}\right|\cdot  \left|\frac{d_2}{d_1}\right|^{\frac{1}{2}}\sum_{c>0}\frac{K_{\frac{1}{2}}\left(d_1,\left(\frac{m}{n}\right)^2d_2;c\right)}{4c}
\begin{cases}
I_{k-\frac{1}{2}}\left(\frac{\pi}{c}\left|\frac{m}{n}\right|\sqrt{\left|d_1d_2\right|}\right)&\text{if }d_1d_2<0,\\
J_{k-\frac{1}{2}}\left(\frac{\pi}{c}\left|\frac{m}{n}\right|\sqrt{d_1d_2}\right)&\text{if }d_1d_2>0.
\end{cases}
$$
We conclude that 
\begin{equation}\label{eqn:Trrewrite}
a_n(k) = (-1)^{\left\lfloor 1-\frac{k}{2}\right\rfloor} \left|m\right|^{k-\frac{3}{2}}\left|d_1\right|^{\frac{3}{4}}\left|d_2\right|^{-\frac{1}{4}} \left|n\right|^{\frac{1}{2}} b_n.
\end{equation}
Using \eqref{eqn:Wspecial} and \eqref{eqn:Fmdef}, we rewrite
$$
b_n=\left(\frac{n^2}{m^2}\left|\frac{d_1}{d_2}\right|\right)^{\frac{k}{2}-\frac{3}{4}} c_{\cM_{d,n}}^{\varepsilon} \left(\varepsilon\left|d_1\right|\right).
$$
Thus from \eqref{eqn:Trrewrite} we obtain
$$
a_n(k) =(-1)^{\left\lfloor 1-\frac{k}{2}\right\rfloor}\left|d_1\right|^{\frac{k}{2}}\left|d_2\right|^{\frac{1-k}{2}} \left|n\right|^{k-1}c_{\cM_{d,n}}^{\varepsilon} \left(\varepsilon\left|d_1\right|\right).
$$
From \eqref{eqn:Rlift}, we have 
\begin{align*}
\widetilde{\Tr}_{d_1,d}\left( R_{2-2k}^{k-1}\left(F_{m,2-2k}\right)\right)&= (-1)^{\left\lfloor 1-\frac{k}{2}\right\rfloor}\left(4\pi\right)^{k-1}\left|d_1\right|^{\frac{k}{2}}\left|d\right|^{\frac{1-k}{2}} \sum_{n\mid m} \left(\frac{d}{n}\right) \left|n\right|^{k-1}c_{\cM_{d,n}}^{\varepsilon}\left(\varepsilon\left|d_1\right|\right),\\
\widetilde{\Tr}_{d_1,d}^*\left(F_{m,2-2k}\right)&=\sum_{n\mid m} \left(\frac{d}{n}\right) \left|n\right|^{k-1}c_{\cM_{d,n}}^{\varepsilon}\left(\varepsilon\left|d_1\right|\right)=c_{\cM_d}^{\varepsilon}\left(\varepsilon\left|d_1\right|\right).
\end{align*}
We have thus shown that the non-holomorphic part of $\widetilde{\Za}_d\left(F_{m,2-2k}\right)$ satisfies Theorem \ref{thm:Zaddef} and its holomorphic part matches that of $\Za_d\left(F_{m,2-2k}\right)$.  This concludes Theorem \ref{thm:Zaddef}, once we have shown that $\Za_d=\widetilde{\Za}_d$.  To prove this, we use Theorem \ref{thm:tildeprop} to compute the non-holomorphic part of $\cM_d$, which was defined in \eqref{eqn:Mdexp}.  Since the coefficients $c_{\cM_d}^-(\delta)$ are completely determined by $\xi_{\frac{3}{2}-k}\left(\cM_d\right)$, it suffices to evaluate this with \eqref{eqn:tildeprop}.

Denoting the $\delta$-th coefficient of $\shin_d\left(\xi_{2-2k}\left(F_{m,2-2k}\right)\right)$ by $r_{m}(\delta)$, \eqref{eqn:xifourier} and \eqref{eqn:tildeprop} yield
$$
c_{\cM_d}^-(\delta)=\frac{(-1)^{\left\lfloor \frac{k}{2}\right\rfloor}2^{k-1}\Gamma\left(k-\frac{1}{2}\right)}{3\left(4\pi|\delta|\right)^{k-\frac{1}{2}}} r_{m}(\delta).
$$
However, by the definition \eqref{eqn:shindef} of the Shintani lifts, one has 
$$
r_{m}(\delta)=\left(d\delta\right)^{\frac{k-1}{2}}\sum_{Q\in \SL_2(\Z)\backslash \QQ_{d\delta}} \chi\left(Q\right)\int_{C_Q} \CC\left(\xi_{2-2k}\left(F_{m,2-2k}\right);Q\right).
$$
This yields \eqref{eqn:Zadeq} and Theorem \ref{thm:Zaddef}.

We next assume that $d_2=D$ satisfies $(-1)^kD<0$.  In this case, we denote
$$
g_{D,n}:=F_{\left(\frac{m}{n}\right)^{2}\left|D\right|,k+\frac{1}{2}}.
$$
Similarly to \eqref{eqn:Kloosterman}, we have
$$
K_{\frac{1}{2}}\left(d_1, \left(\frac{m}{n}\right)^2 D;c\right) = K_{k+\frac{1}{2}}\left(-\left(\frac{m}{n}\right)^2 \left|D\right|,\left(-1\right)^{k}d_1;c\right).
$$
For ease of notation, we abbreviate
$$
b_n:=b_{-\left(\frac{m}{n}\right)^2\left|D\right|,k+\frac{1}{2}}\left(\frac{k}{2}+\frac{1}{4};\left(-1\right)^{k}d_1\right).
$$
By \eqref{eqn:anval} and \eqref{eqn:bcoeff}, we have
$$
a_n(k)=(-1)^{\left\lfloor \frac{k+1}{2}\right\rfloor}\left|m\right|^{k-\frac{3}{2}}\left|d_1\right|^{\frac{3}{4}}\left|d_2\right|^{-\frac{1}{4}}|n|^{\frac{1}{2}}b_n.
$$
For $(-1)^kd_1>0$, \eqref{eqn:Wspecial2} and \eqref{eqn:Fmdef} imply that
$$
b_n=\left(\frac{m^2}{n^2}\left|\frac{d_2}{d_1}\right|\right)^{\frac{k}{2}+\frac{1}{4}} c_{g_{D,n}}^+\left(|d_1|\right).
$$
Therefore, we have
$$
a_n(k)=(-1)^{\left\lfloor \frac{k+1}{2}\right\rfloor}|m|^{2k-1}|n|^{-k}\left|d_2\right|^{\frac{k}{2}}\left|d_1\right|^{\frac{1-k}{2}}c_{g_{D,n}}^+\left(|d_1|\right).
$$
We hence obtain 
\begin{align*}
\Tr_{d_1,D}\left( R_{2-2k}^{k-1}\left(F_{m,2-2k}\right)\right)&= (-1)^{\left\lfloor \frac{k+1}{2}\right\rfloor}\left(4\pi\right)^{k-1}\left|d_1\right|^{\frac{1-k}{2}}\left|D\right|^{\frac{k}{2}} |m|^{2k-1}\sum_{n\mid m} \left(\frac{D}{n}\right) \left|n\right|^{-k}c_{g_{D,n}}^{+}\left(\left|d_1\right|\right),\\
\Tr_{d_1,d}^*\left(F_{m,2-2k}\right)&=|m|^{2k-1}\sum_{n\mid m} \left(\frac{D}{n}\right) \left|n\right|^{-k}c_{g_{D,n}}^{+}\left(\left|d_1\right|\right).
\end{align*}
Finally, we consider the case  $(-1)^kd_1\leq 0$.  By \eqref{eqn:Wspecial2}, we trivially obtain 
$$
c_{g_{D,n}}^{\varepsilon}\left((-1)^kd_1\right)=
\begin{cases}
1&\text{if }d_1=\frac{m^2}{n^2}d_2\text{ and }\varepsilon=1,\\
0 &\text{otherwise.}
\end{cases}
$$
We hence conclude that
$$
\Za_D\left(F_{m,2-2k}\right) = |m|^{2k-1}\sum_{n\mid m} |n|^{-k}g_{D,n}.
$$
Since $g_{D,n}$ is a constant multiple of $P_{(-1)^kD\left(\frac{m}{n}\right)^2,k+\frac{1}{2}}$ by \eqref{eqn:PmClassical} and \eqref{eqn:Fmdef}, the claim follows.
\end{proof}

\subsection{Orthogonality}
In this section, we show the orthogonality of $\Za_D\left(\cM\right)$ to cusp forms.
\begin{proposition}\label{prop:orthog}
If $\cM\in \Hcusp$, then all of the Zagier lifts $\Za_D\left(\cM\right)$ are orthogonal to cusp forms and have vanishing constant terms.  In particular, we have 
\begin{equation}\label{eqn:ZaDeq}
\Za_D\left(\cM\right)=\widetilde{\Za}_D\left(\cM\right).
\end{equation}
\end{proposition}
\begin{proof}
By uniqueness, we only need to prove the first statement.  Since the constant term of $\Za_D(\cM)$ vanishes by definition, it remains to prove that $\Za_D(\cM)$ is orthogonal to cusp forms.  However, Proposition \ref{prop:Zmodular} implies that $\Za_D(\cM)$ is contained in the space spanned by the Poincar\'e series $P_{-m,k+\frac{1}{2}}$ $(m\in \N)$.  It is hence enough to prove that $P_{-m,k+\frac{1}{2}}$ is orthogonal to cusp forms.  For $g(\tau)=\sum_{n=1}^{\infty}a_g(n)q^n\in\SS_{k+\frac{1}{2}}$ we now compute $\left(P_{-m,k+\frac{1}{2}},g\right)_{\reg}$.  In order to evaluate the inner product, we use \eqref{eqn:PmClassical}, \eqref{eqn:Fmdef}, and \eqref{eqn:xiact} to see that there exists a non-zero constant $C_k$ such that 
\begin{equation}\label{eqn:xiF}
\xi_{\frac{3}{2}-k}\left(F_{-m,\frac{3}{2}-k}\right) = C_k P_{-m,k+\frac{1}{2}}.
\end{equation}
Furthermore, $F_{-m,\frac{3}{2}-k}$ has an expansion of the type
\begin{equation}\label{eqn:Fexpansion}
F_{-m,\frac{3}{2}-k}(\tau) = \Gamma\left(k-\frac{1}{2};-4\pi my\right)q^{m} + \sum_{\begin{subarray}{c}n<0 \end{subarray}} c^-(n) \Gamma\left(k-\frac{1}{2}; 4\pi |n| y \right)q^{n} + \sum_{n \geq 0} c^+(n) q^{n}.
\end{equation}
Proceeding as in Bruinier--Funke \cite{BruinierFunke}, we use Stokes' Theorem to show that 
\begin{multline*}
\left\{F_{-m,\frac{3}{2}-k},g\right\}=C_k\left(P_{-m,k+\frac{1}{2}},g\right)_{\reg}=\frac{C_k}{6}\int_{\FF_{\mathcal{T}}(4)} g(\tau) \overline{P_{-m,k+\frac{1}{2}}(\tau)} y^{k-\frac{3}{2}}dx dy \\
= \frac{1}{6} \int_{\FF_{\mathcal{T}}(4)}g(\tau)\overline{\xi_{\frac{3}{2}-k}\left(F_{-m,\frac{3}{2}-k}(\tau)\right)} y^{k-\frac{3}{2}}dx dy= -\frac{1}{6}\int_{\partial \FF_{\mathcal{T}}(4)}g(\tau)\overline{F_{-m,\frac{3}{2}-k}(\tau)} d\tau.
\end{multline*}
A standard argument then shows that the integral along the boundary cancels except for the integral along the line from $-\frac{1}{2}+i\mathcal{T}$ to $\frac{1}{2}+i\mathcal{T}$ and its image under the coset representatives $I$, $ST$, $S$, $ST^{-1}$, $ST^{-2}$, and $ST^{-2}S$, where $T:=\left(\begin{smallmatrix}1&1\\ 0 &1\end{smallmatrix}\right)$ and $S:=\left(\begin{smallmatrix}0&-1\\ 1 &0\end{smallmatrix}\right)$.  Using \eqref{eqn:cuspeven} and \eqref{eqn:cuspodd}, this reduces the proposition to showing the vanishing of 
$$
\frac{1}{6} \int_{-\frac{1}{2}}^{\frac{1}{2}} g\left(x+i\mathcal{T}\right)F_{-m,\frac{3}{2}-k}\left(x+i\mathcal{T}\right)dx.
$$
Every Fourier coefficient of the integrand vanishes except for the constant term.   Using \eqref{eqn:Fexpansion} and the expansion for $g$, the constant term equals
$$
\sum_{n>0} a_g(n) c^-(-n)\Gamma\left(k-\frac{1}{2};4\pi n \mathcal{T}\right),
$$
which vanishes as $\mathcal{T}\to \infty$.  Therefore $P_{-m,k+\frac{1}{2}}$ is orthogonal to cusp forms.

\end{proof}

\subsection{Bijectivity of the lifts}
Finally, we show that the lifts are bijections.

\begin{proposition}\label{prop:bij}
For $k>1$ the maps $\Za_d:\Hcusp\to \HHcusp$ and $\Za_D:\Hcusp\to \SSperpD{D}$ are bijections.
\end{proposition}
\begin{proof}
Injectivity follows from the linearity of the lifts and the uniqueness of functions with given principal parts.

In order to show surjectivity, we recursively construct $\cM\in \Hcusp$ such that the function $\Za_d(\cM)$ (resp. $\Za_D(\cM)$) has a given  principal part.  This is sufficient, since $\SS_{\frac{3}{2}-k}=\{0\}$ and $\Za_D\left(\cM\right)$ is orthogonal to cusp forms by Proposition \ref{prop:orthog}.

By \eqref{eqn:ZaDPoinc} and $\Za_D=\widetilde{\Za}_D$, the principal part of $\Za_D\left(F_{1,2-2k}\right)$ equals $q^{-|D|}$.  We iteratively take linear combinations of $F_{n,2-2k}$, $n\leq m$ in order to obtain a lift with principal part $q^{-|D|m^2}$.  Using \eqref{eqn:ZadPoinc}, the argument for $\Za_d$ follows similarly. 
\end{proof}

\section{A weakly holomorphic Shintani lift}\label{sec:weakShin}
In this section we use the Zagier lifts to construct Shintani lifts for weakly holomorphic modular forms.  For this, we first (uniquely) decompose $f\in S_{2k}^!$ as $f=f_0+f_1$, where $f_0\in S_{2k}$ and $f_1\in \Sperp$.  For each pair $d,D$, the \begin{it}weakly holomorphic Shintani lift\end{it} is then defined by 
\begin{equation}\label{eqn:shindH}
\shin_{d,D}(f):= \xi_{\frac{3}{2}-k}\circ \Za_d\circ \xi_{2-2k}^{-1}\left(f_0\right) +\Za_D\circ \left(\DD^{2k-1}\right)^{-1}\left(f_1\right).
\end{equation}
\begin{remarks}
Before proving Theorem \ref{thm:proportion}, a few comments about the definition \eqref{eqn:shindH} are in order. 
\noindent

\noindent
\begin{enumerate}
\item
 Firstly, note that although $\xi_{2-2k}$ is not injective, one sees that the resulting map is well-defined by using \eqref{eqn:tildeprop}.  No such problem arises when restricting to $\Sperp$, since $\DD^{2k-1}$ is a bijection.
\item
One can extend the definition \eqref{eqn:shindH} to include the case when $d$ and $D$ are not fundamental.  In order to do so, for $\delta=\Delta m^2$ with $\Delta$ fundamental, one sets
\begin{equation}\label{eqn:Zanonfund}
\Za_{\delta}:=T_{m^2}\circ \Za_{\Delta}=\Za_{\Delta}\circ T_m.
\end{equation}
Since $\shin_{\delta}=T_{m^2}\circ\shin_{\Delta}$ by the remark following Theorem 3 of \cite{KohnenMathAnn}, this is the natural definition.
\item
The definition \eqref{eqn:shindH} may at first seem naive, since $f$ is split into two pieces and then separate operators are applied to each piece.  However, we shall see that the lift is naturally tied together via the classical Shintani lift, motivating its name.  
\begin{itemize}
\item[(i)]
The lift maps a cusp form $f$ to a (non-zero) constant multiple of its Shintani lift.  The explicit constant follows by a short calculation and is given in \eqref{eqn:shindD}.
\item[(ii)]
  Suppose that $\cM$ is a preimage of $f\in S_{2k}$ under the $\xi$-operator, and write $g:=\mathcal{D}^{2k-1}(\cM)$.  Although one does not see an immediate connection between $\shin_{d,D}(g)$ and the Shintani lift for one choice of $D$, packaging the $d$-th coefficient of all such lifts (using \eqref{eqn:Zanonfund} for non-fundamental discriminants) into the generating function
\begin{equation}\label{eqn:shingen}
\sum_{\delta:d\delta<0} c_{\shin_{d,\delta}(g)}^+(d) q^{|\delta|}
\end{equation}
yields a constant multiple of $\shin_d(f)$.  This follows by the duality \eqref{eqn:Derduality}.
\item[(iii)]
Moreover, the connection between the two parts is $p$-adically justified.  The authors plan to address this in the forthcoming paper \cite{BGK}.
\end{itemize}
\end{enumerate}
\end{remarks}

\begin{proof}[Proof of Theorem \ref{thm:proportion}]
Theorem \ref{thm:proportion} (1) follows from Lemma \ref{lem:Hecke}, \eqref{eqn:Zadeq}, and \eqref{eqn:ZaDeq}, since $\xi_{2-2k}$ and $\DD^{2k-1}$ are Hecke equivariant.

If $f\in \Sperp$, then we see by \eqref{eqn:shindH} that $\shin_{d,D}(f)$ is in the image of $\Za_D$, which is orthogonal to cusp forms by Proposition \ref{prop:orthog}.  Since Shintani's lift maps cusp forms to cusp forms, Theorem \ref{thm:proportion} (2) follows by \eqref{eqn:shindD}.

\end{proof}
\begin{proof}[Proof of Corollary \ref{cor:weakly}]
Note that by \eqref{eqn:shindH}, for a Hecke eigenform $f$, we have $\shin_{d,D}\left(f\right)=0$ if and only if 
$$
\Za_d\left(\xi_{2-2k}^{-1}\left(f\right)\right)\in \MM_{\frac{3}{2}-k}^!.
$$
By Theorem \ref{thm:proportion} (2), $\shin_{d,D}\left(f\right)=0$ if and only if $\shin_{d}\left(f\right)=0$.  However, it is known that $\shin_d\left(f\right)=0$ if and only if $L\left(f,\chi_d,k\right)=0$ \cite{KohnenZagier,Waldspurger}.
\end{proof}

\begin{proof}[Proof of Theorem \ref{thm:shinimage}]
By Theorem 1.2 of \cite{BruinierOnoRhoades}, the map $\DD^{2k-1}$ is a bijection from $\Hcusp$ to $\Sperp$, while by Proposition \ref{prop:bij}, $\Za_D$ is a bijection between $\Hcusp$ and $\SSperpD{D}$. 
\end{proof}

\section{Proof of Theorems \ref{thm:HeckeD} and \ref{thm:Derduality}}
\subsection{Weakly holomorphic Hecke eigenforms}\label{sec:Hecke}
In this section we prove that $\SSperp\Big/ \bigoplus_D  \JD$, where $\JD=\Za_D\left(S_{2-2k}^!\right)$, is spanned by Hecke eigenforms.  

In order for Theorem \ref{thm:HeckeD} to be meaningful, we must first conclude that the Hecke algebra acts on the relevant subspaces.
\begin{lemma}\label{lem:hecke}
The spaces $\SS_{k+\frac{1}{2}}$, $\Jd$, $\JD$, $\SSperpD{D}$, and $\HHcusp$ are Hecke stable.
\end{lemma}
\begin{proof}
The first statement is classical \cite{KohnenMathAnn}.  Since $\Za_d=\widetilde{\Za}_d$ and $\Za_D=\widetilde{\Za}_D$ by \eqref{eqn:Zadeq} and \eqref{eqn:ZaDeq}, \eqref{eqn:ZagierHeckeD} and \eqref{eqn:ZagierHecked} imply that the subspaces $\JD$ and $\Jd$ are preserved under the Hecke algebra.  Since $\SSperp$ is Hecke stable by \eqref{eqn:TnPet}, the statements for $\SSperpD{D}$ and $\HHcusp$ follow by determining the action of $
T\left(n^2\right)$ on principal parts.
\end{proof}

\begin{proof}[Proof of Theorem \ref{thm:HeckeD}]
By Theorem \ref{thm:Zaharmonic}, we know that $\Za_D$ yields a Hecke equivariant isomorphism between $\Hcusp$ and $\SSperpD{D}$.  Therefore the associated quotient map $\pi_{\Za_{D}}:\Hcusp\to\SSperpD{D}\Big/\JD$ is surjective.  By Lemma \ref{lem:hecke}, $\JD$ is Hecke stable, and hence we have
$$
\Hcusp\Big/\ker\left(\pi_{\Za_{D}}\right)\iso \SSperpD{D}\Big/\JD,
$$
as Hecke modules.  Since $\JD = \Za_D\left(S_{2-2k}^!\right)$, it follows immediately that 
$$
\ker\left(\pi_{\Za_{D}}\right)=S_{2-2k}^!.
$$
However, by Theorem 1.2 of \cite{BGKO}, 
$$
\Hcusp\Big/S_{2-2k}^!\overset{\DD^{2k-1}}{\iso} \Sperp\Big/\DD^{2k-1}\left(S_{2-2k}^!\right)\iso M_{2k}.
$$
Since $\DD^{2k-1}$ and $\Za_D$ are Hecke equivariant, it follows that
$$
M_{2k}\iso \Hcusp\Big/S_{2-2k}^!\overset{\Za_D}{\iso}\SSperpD{D}\Big/\JD,
$$
as Hecke modules.  Finally, Kohnen \cite{Kohnen} proved that  $M_{2k}$ and $\MM_{k+\frac{1}{2}}$ are isomorphic as Hecke modules.  This completes the proof.
\end{proof}

\subsection{The fractional derivatives}\label{sec:halfderiv}
The goal of this section is to prove Theorem \ref{thm:Derduality} concerning the properties of $\Der=\Za_D\circ \Za_{d}^{-1}=\Za_D\circ\widetilde{\Za}_d^{-1}$, where we have used the fact that $\Za_d=\widetilde{\Za}_d$ by uniqueness. 
\begin{proof}[Proof of Theorem \ref{thm:Derduality}]
By Theorem \ref{thm:Zaharmonic}, one sees immediately that $\Der$ is an isomorphism from $\HHcusp$ to $\SSperpD{D}$.  By \eqref{eqn:ZagierHeckeD} and \eqref{eqn:ZagierHecked} together with Proposition \ref{prop:Zmodular}, the lift $\Der$ is Hecke equivariant.  It remains to show \eqref{eqn:Derduality}.   Note that by \eqref{eqn:ZadPoinc}, 
$$
\Za_d\left(F_{1,2-2k}\right) = F_{|d|,\frac{3}{2}-k}.
$$
Similarly, one computes
$$
\Za_D\left(F_{1,2-2k}\right) = P_{-|D|,k+\frac{1}{2}}.
$$
It follows that 
\begin{equation}\label{eqn:DerPoincmap}
\Der\left(F_{|d|,\frac{3}{2}-k}\right)= P_{-|D|,k+\frac{1}{2}}.
\end{equation}
By definition, the constant term of $\Za_D\left(F_{1,2-2k}\right)=P_{-|D|,k+\frac{1}{2}}$ vanishes and Proposition \ref{prop:orthog} implies that $\Za_D\left(F_{1,2-2k}\right)\in \SS_{k+\frac{1}{2}}^{\perp}$.  Hence, since $\xi_{\frac{3}{2}-k}\left(F_{|d|,\frac{3}{2}-k}\right)$ is a cusp form, we see by \eqref{eqn:pairingval} that 
$$
b\left(D,d\right)+a(d,D)=\left\{ P_{-|D|,k+\frac{1}{2}},F_{ |d|,\frac{3}{2}-k}\right\} = \left(\Za_D\left(F_{1,2-2k}\right),\xi_{\frac{3}{2}-k}\left(F_{ |d|,\frac{3}{2}-k}\right)\right)_{\reg}=0. 
$$
\end{proof}

\end{document}